%% file: EquidistributionMatings.tex
\documentclass[letterpaper,10pt]{amsart}
\usepackage[english]{babel}
\usepackage{indentfirst} 
\usepackage{amssymb}
\usepackage{amsmath}
\usepackage{mathabx} 
\usepackage{amsthm}
\usepackage{thmtools}
\usepackage{mathtools}
\usepackage{multirow} 
\usepackage{bbm}             
\usepackage[usenames,dvipsnames]{xcolor} 
\usepackage{comment}
\usepackage{graphicx}
\usepackage[all]{xy} 

\newtheorem{theorem}{Theorem}[section]
\newtheorem{definition}[theorem]{Definition}
\newtheorem{proposition}[theorem]{Proposition}

\newtheorem{lemma}[theorem]{Lemma}

\newtheorem{remark}[theorem]{Remark}

\newtheorem*{theorem-non}{Theorem}

\theoremstyle{definition}

\def\PSL{\operatorname{PSL}}

\def\cov{\operatorname{Cov^Q_0}}
\def\supp{\operatorname{supp}}
\def\Per{\operatorname{Per}}
\def\critval{\operatorname{CritVal}}
\def\critpt{\operatorname{CritPt}}

\def\graph{\operatorname{Gr}}
\def\bif{\operatorname{BIF}}
\def\interior{\operatorname{int}}

\newcommand\restr[2]{{
  \left.\kern-\nulldelimiterspace 
  #1 
  \vphantom{\big|} 
  \right|_{#2} 
  }}

\def\F{\mathcal{F}_a}
\def\J{\operatorname{J}_a}
\def\m{\mu_+}
\def\E{\mathcal{E}_a}
\def\f{f_a}

\usepackage{pgfplots}
\pgfplotsset{compat=1.15}
\usepackage{mathrsfs}
\usetikzlibrary{arrows}
\begin{document}

\title[Equidistribution of $\F$]{Equidistribution for matings of quadratic maps with the Modular group}
\date{\today}
\author[V. Matus de la Parra]{V. Matus de la Parra}

\address{718 Hylan Building, University of Rochester}

\begin{abstract}
We study the asymptotic behavior of the family of holomorphic correspondences $\lbrace\F\rbrace_{a\in\mathcal{K}}$, given by
$$\left(\frac{az+1}{z+1}\right)^2+\left(\frac{az+1}{z+1}\right)\left(\frac{aw-1}{w-1}\right)+\left(\frac{aw-1}{w-1}\right)^2=3.$$
It was proven by Bullet and Lomonaco that $\F$ is a mating between the modular group $\PSL_2(\mathbb{Z})$ and a quadratic rational map. We show for every $a\in\mathcal{K}$, the iterated images and preimages under $\F$ of nonexceptional points equidistribute, in spite of the fact that $\F$ is weakly-modular in the sense of Dinh, Kaufmann and Wu but it is not modular. Furthermore, we prove that periodic points equidistribute as well.
\end{abstract}
\maketitle

\section{Introduction}
\input{Introduction}

\section{Preliminaries}
\subsection{Holomorphic Correspondences}\label{hol}
\input{Preliminaries}

\subsection{The Family $\lbrace\F\rbrace_a$}\label{thefamily}
\input{Family}

\subsection{Critical Values}\label{critical}
\input{Critical}

\subsection{Klein Combination pairs}\label{kleincomb}
\input{Klein}

\section{Clasification of the Family $\lbrace\F\rbrace_a$}\label{weaklymodular}
\input{WM}

\section{Exceptional Set and Periodic Points}\label{exceptional and periodic}
\input{Exceptional}

\section{Equidistribution of $\lbrace\F^n\rbrace_n$}\label{equidistribution}
\input{Equid}

\bibliographystyle{plain}
\bibliography{References}
\end{document}

%% file: Introduction.tex
In 1965, Brolin \cite{Bro65} studied asymptotic properties of polynomials $P(z)\in\mathbb{C}[z]$ of degree bigger than or equal to $2$. He proved the existence of a probability measure for which the preimages $P^{-n}(z_0)$ at time $n$ of any point $z_0\in\mathbb{C}$ (with at most one exception) asymptotically equidistribute, as $n$ tends to infinity. In 1983, Freire, Lopes and Ma\~n\'e \cite{FreLopMan83}, and Ljubich \cite{Lju83}, independently proved the generalization to rational maps of degree at least $2$ on the Riemann sphere. These results have been generalized to different settings. For instance, see \cite[Section 1.4]{DinSib10} and the references therein for higher dimensions, and \cite{FavRiv10} and \cite{Gig14} for the non-archimedian setting.

The equidistribution properties of holomorphic correspondences have also attracted considerable interest. Roughly speaking, a \emph{holomorphic correspondence} on a complex manifold $X$ is a multivalued map induced by a formal sum $\Gamma=\sum n_i\Gamma(i)$ of complex varieties $\Gamma(i)\subset X\times X$ of the same dimension. The multivalued map sends $z$ to $w$ if $(z,w)$ belongs to some $\Gamma(i)$ (see Section \ref{hol}). Let $d(F)$ denote the number of pre-images of a generic point under $F$. We call this number the \emph{topological degree} of $F$, just as in the case of rational maps. We study the existence of a Borel probability measure $\mu$ on $X$ that has the property that for all but at most finitely many $z_0\in X$,
$$\frac{1}{d(F)^n}(F^n)_*\delta_{z_0}\to\mu,$$
as $n\to\infty$. Here $(F^n)_*$ denotes the pushforward operator associated to $F^n$. In \cite{Din05}, Dinh studied the case of polynomial correspondences whose Lojasiewicz exponent is strictly bigger than $1$, in which case we always have that $d(F^{-1})<d(F)$. The case where $d(F)=d(F^{-1})$ is open but some cases are known. For instance, Clozel, Oh and Ullmo \cite{CloOhUll01} proved equidistribution for irreducible modular correspondences, Clozel and Otal \cite{CloOta01} proved it for exterior modular correspondences, and Clozel and Ullmo \cite{CloUll03} for those that are self-adjoint. On the other hand, Dinh, Kaufmann and Wu proved in \cite{DinKauWu20} that if such $F$ is not weakly-modular, then the statement holds for both $F$ and $F^{-1}$. Modular correspondences are weakly-modular, but the reverse containment does not hold. On a different classification, Bharali and Sridharan \cite{bhasri16} proved equidistribution for correspondences with $d(F)\geq d(F^{-1})$ having a repeller in the sense of \cite{McG92}. We will study a $1$-parameter family in the gap between weak-modularity and modularity, and for which the result in \cite{bhasri16} does not apply.
\\

Our object of study is the family of correspondences $\lbrace \F\rbrace_{a\in\mathcal{K}}$ on the Riemann sphere $\widehat{\mathbb{C}}$, where $\F$ is given in affine coordinates by
\begin{equation}\label{equ}
\left(\frac{az+1}{z+1}\right)^2+\left(\frac{az+1}{z+1}\right)\left(\frac{aw-1}{w-1}\right)+\left(\frac{aw-1}{w-1}\right)^2=3,
\end{equation}
and $\mathcal{K}$ is the \emph{Klein combination locus} defined in Section \ref{thefamily}. This family was studied by Bullett et al. in \cite{BulHar00}, \cite{BulLom20I}, \cite{BulLom20II} and \cite{BulPen94}. In \cite{BulLom20I}, Bullett and Lomonaco proved that there is a two sided restriction $\f$ of $\F$ that is hybrid equivalent to a quadratic rational map $P$ that has a fixed point with multiplier $1$. We refer the reader to \cite{Lom15} for conjugacy of parabolic-like mappings. Moreover, for the parameters for which the Julia set of $P$ is connected, we have that $\F$ is a mating between the rational map $P$ and the modular group $\PSL_2(\mathbb{Z})$. This generalizes a previous result by Bullett and Penrose \cite{BulPen94}. The correspondence $\F$ has two homeomorphic copies of $K_P$, denoted $\Lambda_{a,-}$ and $\Lambda_{a,+}$, and they satisfy that $\F^{-1}(\Lambda_{a,-})=\Lambda_{a,-}$ and $\F(\Lambda_{a,+})=\Lambda_{a,+}$. These are called \emph{backward} and \emph{forward limit set}, respectively (see Section \ref{limitsets}).
\\

The following theorem states that this family does not fit the conditions for any of the equidistribution results listed above (see Section \ref{notes}).

\begin{theorem}\label{A} For every $a\in\mathcal{K}$, we have that
\begin{itemize}
\item[1.] $\F$ is a weakly-modular correspondence that is not modular, and
\item[2.]\label{nonrepeller} $\partial\Lambda_{a,-}$ is not a repeller for $\F$. 
\end{itemize}
\end{theorem}

Furthermore, we prove that $\F$ satisfies a property that is stronger than weak-modularity (see Remark \ref{norm}).

The purpose of this paper is to show that equidistribution holds for the family $\lbrace\F\rbrace_{a\in\mathcal{K}}$. Put
$$\E\coloneqq\left\lbrace\begin{matrix}
\emptyset & \mbox{ if }a\neq 5\\
\lbrace -1,2\rbrace & \mbox{ if }a=5.\end{matrix}\right.$$
We prove the following equidistribution theorem.

\begin{theorem}\label{B}
Let $a\in\mathcal{K}$. There exist two Borel probability measures $\m$ and $\mu_-$ on $\widehat{\mathbb{C}}$, with $\supp(\m)=\partial\Lambda_{a,+}$ and $\supp(\mu_-)=\partial\Lambda_{a,-}$, such that for every $z_0\in\widehat{\mathbb{C}}\setminus\E$,
$$\frac{1}{2^n}(\F^n)_*\delta_{z_0}\to\m\mbox{ and }\frac{1}{2^n}(\F^{-n})_*\delta_{z_0}\to\mu_-,$$
weakly, as $n\to\infty$.
\end{theorem}

In later work \cite{MdlP22II} we prove that the measures $\mu_+$ and $\mu_-$ maximize entropy: the metric entropy in \cite{VivSir22} yields equality for the Half-Variational Principle with the topological entropy in \cite{DinSib08II}.

Let $F$ be a holomorphic correspondence on $X$ with graph $\Gamma$. Denote by $\Gamma^{(n)}$ the graph of $F^n$ and by 
$$\mathfrak{D}_{X}\coloneqq\lbrace (z,z)|z\in X\rbrace$$
the diagonal in $X\times X$. Then the set of \emph{periodic points of $F$ of period $n$} is defined as the set
$$\Per_n(F)\coloneqq\pi_1\left(\Gamma^{(n)}\cap\mathfrak{D}_{X}\right),$$
and for $z\in\Per_n(F)$, we define the \emph{multiplicity of $z$ as a periodic point of $F$ of order $n$} to be the number $\nu_{\restr{\pi_1}{\Gamma^{(n)}\cap\mathfrak{D}_X}}(z,z)$, defined in Section \ref{hol}.

Another source of motivation is whether or not periodic points equidistribute. In \cite{Lju83}, Ljubich showed  that this is the case for rational maps of degree bigger than or equal to $2$, where periodic points are counted either with or without multiplicity. The equidistribution of periodic points is also studied in \cite{BriDuv99}, \cite{DinNguTru15} and \cite{FavRiv10} in the case of maps, and in \cite{Din05}, \cite{Din13} in the case of correspondences. We prove this holds for the family $\lbrace\F\rbrace_{a\in\mathcal{K}}$ as well.

\begin{theorem}\label{C} For $a\in\mathcal{K}$,
$$\frac{1}{|\Per_n(\F)|}\sum\limits_{z\in\Per_n(\F)}\delta_z\hspace{.5cm}\mbox{and}\hspace{.5cm}\frac{1}{2^{n+1}}\sum\limits_{z\in\Per_n(\F)}\nu_{\restr{\pi_1}{\Gamma_a^{(n)}\cap\mathfrak{D}_{\widehat{\mathbb{C}}}}}( z,z)\delta_z$$
are both weakly convergent to $\frac{1}{2}(\mu_-+\mu_+)$, as $n\to\infty$.
\end{theorem}

In Section \ref{equidistribution}, we define the set $\hat{P}^{\Gamma}_n$ of  \emph{superstable parameters of order $n$}. Combining the main results of both \cite{BulLom20II} and \cite{PetRoe21}, we obtain a homeomorphism $\Psi:\mathcal{M}\to\mathcal{M}_\Gamma$ between the Mandelbrot Set $\mathcal{M}$ and the connectedness locus $\mathcal{M}_{\Gamma}$ of the family $\lbrace\F\rbrace_{a\in\mathcal{K}}$. These results, together with the equidistribution result in \cite{Lev90} yield the following theorem.

\begin{theorem}\label{D}
In $\mathcal{M}_\Gamma$, superstable parameters equidistribute with respect to $\Psi^*m_{\bif}$, \emph{i.e.,}
$$\lim\limits_{n\to\infty}\frac{1}{2^{n-1}}\sum\limits_{a\in\hat{P}^{\Gamma}_n}\delta_a=\Psi^*m_{\bif}.$$
\end{theorem}

\subsection{Notes and references}\label{notes}

There is a bigger family studied by Bullett and Harvey in \cite{BulHar00}, given by replacing the right hand side of equation (\ref{equ}) by $3k$, where $k\in\mathbb{C}$. For these correspondences it is also possible to define limit sets $\Lambda_{a,k,-}$ and $\Lambda_{a,k,+}$ analogous to the ones in the case where $k=1$. In \cite{bhasri16}, Bharali and Sridharan show how their equidistribution result applies to these correspondences in the case where $\Lambda_{a,k,-}$ is a repeller. Parameters for which this is the case  exist from the results in \cite{BulHar00}. However, we prove in Theorem \ref{A} part $2$ that this is never the case when $k=1$ and $a\in\mathcal{K}$.

\subsection{Organization}

The structure of this paper is as follows. In Section \ref{hol}, we give an introduction to holomorphic correspondences and their action on Borel measures. In Section \ref{thefamily}  we introduce the correspondences $\F$ given by equation (\ref{equ}). We define critical values and find those of $\F$ in Section \ref{critical}, and define Klein combination pair and the Klein combination locus $\mathcal{K}$ in Section \ref{kleincomb}. In Section \ref{modularity}, we define modular and weakly-modular correspondences, and prove part $1$ of Theorem \ref{A}. In order to prove that $\F$ is weakly-modular, we use the decomposition $\F=\J\circ\cov$ given in \cite{BulLom20I}, into a certain involution $\J$ composed with the deleted covering correspondence $\cov$ described in Section \ref{thefamily}, and construct the measures in the definition of weakly-modular using the symmetry of the graph of $\cov$. The fact that  $\F$ is not modular follows from the fact that Borel measures assigning positive measure to nonempty open sets are not invariant by $\F$. In Section \ref{limitsets}, we define the limit sets $\Lambda_{a,-}$ and $\Lambda_{a,+}$, and prove part $2$ of Theorem \ref{A} by showing that the parabolic fixed point in $\partial\Lambda_{a,-}$ violates the definition of a repeller. In Section \ref{exceptional and periodic}, we describe the exceptional set of the two-sided restriction $\f$ and the set $\Per_n(\F)$ of periodic points of period $n$. Finally, in Section \ref{equidistribution}, we use the description of $\f$ given in \cite{BulLom20I} together with the results in \cite{FreLopMan83} and \cite{Lju83} to prove Theorem \ref{B} and Theorem \ref{C} about asymptotic equidistibution of images, preimages and periodic points. We finish with the proof of Theorem \ref{D} about equidistribution of special points in the Modular Mandelbrot set $\mathcal{M}_\Gamma$.

\textbf{Ackwoledgements:}
The author would like to thank Shaun Bullett and Luna Lomonaco for a useful conversation in Santiago de Chile, and specially thank Juan Rivera-Letelier for his great patience and guidance.

%% file: Preliminaries.tex
Let $X$ be a compact Riemann surface and let $\pi_j:X\times X\to X$ be the canonical projection to the $j$-th coordinate, $j=1,2$. We say that a formal sum $\Gamma=\sum_i n_i\Gamma(i)$ is a \emph{holomorphic 1-chain} on $X\times X$ if its support $\supp\Gamma\coloneqq\bigcup_i \Gamma(i)$ is a subvariety of $X\times X$ of pure dimension $1$ whose irreducible components are exactly the $\Gamma(i)$'s, and the $n_i$'s are non-negative integers. We say that the $\Gamma(i)$'s are the \emph{irreducible components of $\Gamma$}.

Let $\Gamma=\sum_in_i\Gamma(i)$ be a holomorphic 1-chain satisfying that for $j=1,2$ and every $i$ such that $n_i>0$, the restriction $\restr{\pi_j}{\Gamma(i)}$ of the canonical projection $X\times X\to X$ to the irreducible component $\Gamma(i)$ is surjective. The chain $\Gamma$ induces a multivalued map $F$ from $X$ to itself by 
$$F(z)\coloneqq\bigcup\limits_i\restr{\pi_2}{\Gamma(i)}\left(\restr{\pi_1}{\Gamma(i)}^{-1}(z)\right).$$
The multivalued map $F$ is called a \emph{holomorphic correspondence} and it is said to be  \emph{irreducible} if $\sum_i n_i=1$. We say that $\Gamma_F\coloneqq\Gamma$ is the \emph{graph} of the holomorphic correspondence $F$. Let $\iota:X\times X\to X\times X$ be the involution $(z,w)\mapsto(w,z)$. We can define the \emph{adjoint correspondence} $F^{-1}$ of $F$ by the relation $F^{-1}(z)\coloneqq\bigcup_i\restr{\pi_1}{\Gamma(i)}(\restr{\pi_2}{\Gamma(i)}^{-1}(z))$, which is a holomorphic correspondence, whose graph is the holomorphic 1-chain $\Gamma_F^{-1}=\sum_in_i\iota(\Gamma(i))$.

In \cite{Sto66}, Stoll introduced a notion of multiplicity that will be useful for this paper. Let $M$ be a quasi-projective variety and $N$ a smooth quasi-projective variety. If $g:M\to N$ is regular, and $a\in M$, then 
we say that a neighborhood $U$ of $a$ is \emph{distinguished with respect to $g$ and $a$} if $\overline{U}$ is compact and $g^{-1}(g(a))\cap\overline{U}=\lbrace a\rbrace$. Such neighborhoods exist if and only if $\dim_a g^{-1}(g(a))=0$, and in this case they form a base of neighborhoods. If $U$ is distinguished with respect to $g$ and $a$, then put $\mu_g(z,U)\coloneqq|g^{-1}(g(z))\cap U|$. It can be shown that $\hat{\nu}_g(a)\coloneqq\limsup\limits_{z\to a}\mu_g(z,U)$ does not depend on the distinguished neighborhood $U$, and the maps $n_b$ defined on $N$ by $a\mapsto\sum\limits_{b\in g^{-1}(a)}\hat{\nu}_g(b)$ is constant in each component.

Suppose that $g:M\to N$ is a finite and surjective regular map, with $M$ and $N$ as above. Stoll proved in \cite{Sto66} that $\hat{\nu}_g(a)$ generalizes the notion of multiplicity of $g$ at $a$ and whenever $\varphi$ is a continuous function with compact support in $M$, the map
$$a\mapsto\sum\limits_{b\in g^{-1}(a)}\hat{\nu}_g(b)\varphi(b)$$
is continuous.

In order to study dynamics, we proceed to define the \emph{composition} of two holomorphic correspondences $F$ and $G$ with associated holomorphic 1-chains $\Gamma_F=\sum_i n_i\Gamma_F(i)$ and $\Gamma_G=\sum_j m_j\Gamma_G(j)$, respectively. For each $i$ and $j$, let $A_{i,j}$ be the image of the projection $p_{i,j}:(\Gamma_G(j)\times\Gamma_F(i))\cap\lbrace x_2=x_3\rbrace\hookrightarrow X\times X$ that forgets the second and third coordinates, \emph{i.e.},
$$A_{i,j}=\lbrace (z,w)\in X\times X|\exists x\in X\mbox{ such that }(z,x)\in\Gamma_G(j)\mbox{, and }(x,w)\in\Gamma_F(i)\rbrace.$$
Let $\lbrace\Gamma(i,j,k)\rbrace_{k=1}^{N(i,j)}$ be the irreducible components of $A_{i,j}$. Observe that since $\Gamma_G(j)$ and $\Gamma_F(i)$ are both quasi-projective, and so is $\lbrace x_2=x_4\rbrace\subset X^4$, then $p_{i,j}$ is a regular map from the quasi-projective variety $(\Gamma_G(j)\times\Gamma_F(i))\cap\lbrace x_2=x_3\rbrace$ to the smooth quasi-projective variety $X\times X$. Then we have that the map $a\mapsto n_{\restr{p_{i,j}}{\Gamma(i,j,k)}}(a)$ defined on $\Gamma(i,j,k)$ is constant. Therefore, $\eta_{i,j,k}\coloneqq n_{\restr{p_{i,j}}{\Gamma(i,j,k)}}(a)$ denotes the number of $x\in X$ such that $((z,x),(x,w))\in\Gamma_G(j)\times\Gamma_F(i)$, for a generic point $(z,w)\in\Gamma(i,j,k)$. Define the composition $F\circ G$ as the holomorphic correspondence determined by the holomorphic 1-chain
$$\Gamma_{F\circ G}\coloneqq\sum\limits_{i,j}\sum\limits_{k=1}^{N(i,j)}n_im_j\eta_{i,j,k}\Gamma(i,j,k).$$
Note that the $\supp\Gamma_{F\circ G}=\bigcup_{i,j}A_{i,j}$.

Set $d(F)\coloneqq\sum_i n_i\deg(\restr{\pi_2}{\Gamma(i)})$. We have that $d(F\circ G)=d(F)d(G)$. Thus, in particular, for every integer $n\geq 1$ we have that $d(F^n)=(d(F))^n$. We call $d(F)$ the \emph{topological degree} of $F$, and it corresponds to the number of preimages of a generic point under $F$.

If $F$ is an irreducible holomorphic correspondence over $X$ with graph $\Gamma$ and $\varphi:X\to\mathbb{C}$ is a continuous function, then
$$F^*\varphi(z)\coloneqq\sum\limits_{(z,w)\in\restr{\pi_1}{\Gamma}^{-1}(z)}\hat{\nu}_{\restr{\pi_1}{\Gamma}}(z,w)\varphi(w)$$
is continuous as well, see \cite[Lemma 1.1]{CloUll03}. Now let $F$ be a holomorphic correspondence that is not necessarily irreducible, with graph $\Gamma_F=\sum_i n_i\Gamma(i)$. We denote by $F_i$ the holomorphic correspondence induced by $\Gamma(i)$ and we put
$$\nu_{F_i}(z,w)=\left\lbrace
\begin{matrix}
\hat{\nu}_{\restr{\pi_1}{\Gamma(i)}} &\mbox{if }(z,w)\in\Gamma(i)\\ 
0 &\mbox{otherwise} 
\end{matrix}\right.$$
and $\nu_{F}\coloneqq\sum_i n_i\nu_{F_i}$. Then for every continuous function $\varphi:X\to\mathbb{C}$, the map
\begin{eqnarray*}
z\mapsto\sum\limits_{w\in F(z)}\nu_{F}(z,w)\varphi(w)&=&\sum\limits_{w\in f(z)}\left(\sum\limits_i n_i\nu_{F_i}(z,w)\right)\varphi(w)\\
&=&\sum\limits_i n_i {F_i}_*\varphi(w)
\end{eqnarray*}
is also continuous.

The holomorphic correspondence $F$ induces an action $F_*$ on finite Borel measures $\mu$ by duality, namely $\langle F_*\mu,\varphi\rangle\coloneqq\langle\mu,{F}^*\varphi\rangle$, called the \emph{push-forward operator} and the resultant measure $F_*\mu$ is the \emph{push-forward measure of $\mu$ under $F$}. We define as well the action ${F}^*\coloneqq(F^{-1})_*$, called the \emph{pull-back operator} and the resultant measure $F^*\mu$ is called the \emph{pull-back measure of $\mu$ under $F$}. This action on measures agrees with the action on points
$$F^*\delta_z\coloneqq\sum\limits_{w\in F(z)}\nu_{F}(z,w)\delta_w,$$
where $\delta_z$ is the Dirac delta at $z$.

In order to see this, note that for every continuous function $\varphi:X\to\mathbb{C}$
\begin{eqnarray*} 
\left\langle\sum\limits_{w\in F(z)}\nu_{F}(z,w)\delta_w,\varphi\right\rangle&=&\sum\limits_{w\in F(z)}\nu_{F}(z,w)\langle\delta_w,\varphi\rangle\\
&=&\sum\limits_{w\in F(z)}\nu_{F}(z,w)\varphi(w)\\
&=&\int\left(\sum\limits_{w\in F(\zeta)}\nu_{F}(\zeta,w)\varphi(w)\right)d\delta_z(\zeta)\\
&=&\langle\delta_z,F_*\varphi\rangle\\
&=&\langle F^*\delta_z,\varphi\rangle.
\end{eqnarray*}

%% file: Family.tex
Let $Q(z)\in\mathbb{C}[z]$ be a nonlinear polynomial. The \emph{deleted covering relation of $Q$} on $\mathbb{C}\times\mathbb{C}$ is defined by $w\in\cov(z)$ if and only if 
\begin{equation}\label{covering}
P_Q(z,w)\coloneqq\frac{Q(z)-Q(w)}{z-w}=0.
\end{equation}
Note that the denominator ``deletes" the obvious association of $z$ with itself in the equation $Q(z)=Q(w)$.

In this section, we will identify $\widehat{\mathbb{C}}$ with the complex projective line when it is convenient to work with homogeneous coordinates $(z:w)$.

\begin{proposition}\label{holcorr} Put $Q(z)\coloneqq z^3-3z$. The closure of the relation (\ref{covering}) is an irreducible quasiprojective complex variety $\Gamma_0$ of $\widehat{\mathbb{C}}\times\widehat{\mathbb{C}}$ of dimension $1$. Moreover, the projections $\restr{\pi_1}{\Gamma_0}$ and $\restr{\pi_2}{\Gamma_0}$ are both surjective and of degree 2.
\end{proposition}

\begin{proof}
Note that $P_Q(z,w)=z^2+zw+w^2-3$ and consider $P_Q(z,w)$ as a single variable polynomial in $(\mathbb{C}[z])[w]$. Then its discriminant $-3z^2+12$ is not a square in $\mathbb{C}[z]$. Therefore $P_Q(z,w)$ is an irreducible polynomial, and hence
$$\mathcal{Z}\coloneqq\lbrace (z,w)\in\mathbb{C}\times\mathbb{C}\hspace{.1cm}|P_Q(z,w)=0\rbrace$$
is an irreducible subvariety of $\mathbb{C}\times\mathbb{C}$.

Now we want to describe the closure $\overline{\mathcal{Z}}$ of $\mathcal{Z}$ in $\widehat{\mathbb{C}}\times\widehat{\mathbb{C}}$. Observe that if we fix $z\in\mathbb{C}$, then $\lim\limits_{w\to\infty}P_Q(z,w)=\infty$, and if we fix $w\in\mathbb{C}$, then $\lim\limits_{z\to\infty}P_Q(z,w)=\infty$. Therefore there are no points of the form $(z,\infty)$ or $(\infty, w)$ in $\overline{\mathcal{Z}}$. Given $R>0$, let $z\in\mathbb{C}$ be such that $|z|=R$. Observe that $P_Q(z,\cdot)\in\mathbb{C}[w]$ is nonconstant and therefore has at least one root in $\mathbb{C}$. Let $w\in\mathbb{C}$ be a root. Then $(z,w)\in\mathcal{Z}$ and we have that $|w^3-3w|=|Q(w)|=|Q(z)|=|z^3-3z|\geq R^3-3R$. By taking $R\to\infty$, we get that $(z,w)\to(\infty,\infty)$. Therefore $\overline{\mathcal{Z}}=\mathcal{Z}\cup\lbrace (\infty ,\infty)\rbrace$. In particular, $\overline{\mathcal{Z}}$ extends the relation given by (\ref{covering}) from $\mathbb{C}\times\mathbb{C}$ to $\widehat{\mathbb{C}}\times\widehat{\mathbb{C}}$.

Take the homogenization
$$T(z,x,w,y)\coloneqq z^2y^2+zxwy+x^2w^2-3x^2y^2$$
of $P_Q(z,w)$, so $P_Q(z,w)=T(z,1,w,1)$, and note that $T(\lambda_1z,\lambda_1x,\lambda_2w,\lambda_2y)=\lambda_1^2\lambda_2^2T(z,x,w,y)$. Thus, for the closed subvariety
$$\Gamma_0\coloneqq\lbrace((z:x),(w:y))\in\widehat{\mathbb{C}}\times\widehat{\mathbb{C}}\hspace{.1cm}|T(z,x,w,y)=0\rbrace$$
of $\widehat{\mathbb{C}}\times\widehat{\mathbb{C}}$, we have that $\Gamma_0=\overline{\mathcal{Z}}$. In order to prove that $\Gamma_0$ is irreducible, note that each of its irreducible components intersecting $\mathbb{C}\times\mathbb{C}$ must be a closed subset of $\Gamma_0$ containing $\Gamma_0\cap(\mathbb{C}\times\mathbb{C})=\mathcal{Z}$, and therefore it is $\Gamma_0$ itself.
Thus $\Gamma_0$ has only one irreducible component and hence it is irreducible.

We proceed to show $\Gamma_0$ has dimension 1. Observe that the polynomial $T(z,x,w,y)$ is irreducible in $\mathbb{C}[z,x,w,y]$, as whenever $S(z,x,w,y)|T(z,x,w,y)$ in $\mathbb{C}[z,x,w,y]$, then $S(z,1,w,1)|P_Q(z,w)$ in $\mathbb{C}[z,w]$. Therefore the zero set $Z(T)\subset\mathbb{C}^2\times\mathbb{C}^2$ of $T$ is an irreducible hypersurface of $\mathbb{C}^2\times\mathbb{C}^2$, and hence it has codimension 1. Now let $p:\mathbb{C}^2\setminus\lbrace(0,0)\rbrace\to\widehat{\mathbb{C}}$ be the projection sending $(z,x)\mapsto(z:x)$. Note that in the chart
$$U_1=\lbrace (z,w)\in\mathbb{C}^2\setminus\lbrace (0,0)\rbrace|w\neq 0\rbrace,$$
the map $p$ is simply $(z,w)\mapsto(z/w:1)$, and in the chart
$$U_2=\lbrace (z,w)\in\mathbb{C}^2\setminus\lbrace (0,0)\rbrace|z\neq 0\rbrace,$$
it becomes $(z,w)\mapsto(1:w/z)$. Let $\hat{p}:(\mathbb{C}^2\setminus\lbrace(0,0)\rbrace)\times(\mathbb{C}^2\setminus\lbrace(0,0)\rbrace)\to\widehat{\mathbb{C}}\times\widehat{\mathbb{C}}$ be the map defined by $\hat{p}((z,x),(w,y))\coloneqq(p(z,x),p(w,y))$. Then $\restr{\hat{p}}{Z(T)}:Z(T)\to\Gamma_0$ is a regular map between irreducible varieties, and $\restr{\hat{p}}{Z(T)}$ has constant fiber dimension equal to 2, as $T$ is homogeneous in $(z,x)$ and in $(w,y)$. Therefore, 
$$\dim\Gamma_0=\dim Z(T)-\dim\restr{\hat{p}}{Z(T)}^{-1}(z,w)=1.$$

Finally, observe that the polynomial equation (\ref{covering}) has at least 1 and at most 2 solutions for every $z\in\mathbb{C}$, and by symmetry the same holds for $w\in\mathbb{C}$. Note as well that $\infty$ is in correspondence with and only with itself. Thus, the projections $\restr{\pi_1}{\Gamma_0}$ and $\restr{\pi_2}{\Gamma_0}$ are both surjective. Moreover, $P_Q(1,1)=P_Q(1,-2)=0$, and hence $P_Q(1,w)$ has exactly 2 solutions. Therefore $\deg(\restr{\pi_1}{\Gamma_0})=\deg(\restr{\pi_2}{\Gamma_0})= 2$.
\end{proof}

\begin{remark} Proposition \ref{holcorr} holds for a large class of polynomials $Q(z)$. Observe that no polynomial can have $Q(z)=Q(\infty)$ for a finite number $z$. Therefore, following the proof of Proposition \ref{holcorr}, we conclude that in order to get an irreducible holomorphic correspondence, it suffices to prove that $P_Q(z,w)$ is irreducible over $\mathbb{C}$. This holds under fairly general conditions. For instance, this is the case when $Q$ is indecomposable and not linearly related to either $z^n$ or a Chebyshev polynomial (see \cite{Fri70}).
\\
On the other hand, note that whenever $Q=R\circ S$ with $R$ and $S$ of degree greater than $1$, then $P_S(z,w)$ divides $P_Q(z,w)$, and therefore $P_Q(z,w)$ is reducible.
\vspace{.5cm}
\end{remark}

Proposition \ref{holcorr} says that $\Gamma_0$ is the graph of an irreducible holomorphic
correspondence, where $\Gamma_0$ is a quasi-projective variety and
$\restr{\pi_1}{\Gamma_0}:\Gamma_0\to\widehat{\mathbb{C}}$ is a
finite and surjective morphism over $\mathbb{C}$, and hence we can use
our definition of pull-back and push-forward operators induced by the correspondence
on finite measures. We call this correspondence the \emph{deleted covering correspondence of $Q$}, denoted by $\cov$ as well. That is, $\cov$ is the holomorphic correspondence on $\widehat{\mathbb{C}}\times\widehat{\mathbb{C}}$ such that $\Gamma_{\cov}=\Gamma_0$. From now on, we always consider $Q(z)=z^3-3z$.

Now take $a\in\mathbb{C}\setminus\lbrace 1\rbrace$ and let $\J:\widehat{\mathbb{C}}\to\widehat{\mathbb{C}}$ be the involution 
$$\J(z)\coloneqq\frac{(a+1)z-2a}{2z-(a+1)}.$$
The composition of $\cov$ with the involution $\J$ is again an irreducible holomorphic correspondence $\F\coloneqq\J\circ\cov$ on $\widehat{\mathbb{C}}\times\widehat{\mathbb{C}}$ with graph $\Gamma_a$ for which we can use the pull-back and push-forward operators above as well. Note that $d(\F)=d(\F^{-1})=2$ and $\F^{-1}=\cov\circ\J=\J\circ\F\circ\J$, since $\cov^{-1}=\cov$ and $\J^{-1}=\J$.

Set $\phi_a(z)\coloneqq\frac{az+1}{z+1}$. Then $\phi_a^{-1}\circ\F\circ\phi_a$ is a holomorphic correspondence on $\widehat{\mathbb{C}}\times\widehat{\mathbb{C}}$ that restricted to $(\widehat{\mathbb{C}}\setminus\lbrace -1\rbrace)\times(\widehat{\mathbb{C}}\setminus\lbrace1\rbrace)$ induces the relation given by the equation (\ref{equ}). Thus  $(z,w)\in(\widehat{\mathbb{C}}\setminus\lbrace -1\rbrace)\times(\widehat{\mathbb{C}}\setminus\lbrace1\rbrace)$ satisfies equation (\ref{equ}) if and only if $w\in\F(z)$ \cite[Lemma 3.1]{BulLom20I}.

Observe that $\F^{-1}(1)=\cov(\J(1))=\cov(1)=\lbrace 1,-2\rbrace$, independent of $a\in\mathbb{C}\setminus\lbrace 1\rbrace$. In particular, $1\in\Per_1(\F)$ and we say that $1$ is a \emph{fixed point} of $\F$.

%% file: Critical.tex
In this section we will discuss what parts of the graph of $\F$ are locally the graph of a holomorphic function, by defining and finding all critical values and ramification points of $\Gamma_a$. This will be used in Section 4 to find the exceptional set.

\begin{definition}\label{crit} Let $\Gamma$ be the graph of an irreducible holomorphic correspondence on $\widehat{\mathbb{C}}$, and put
$$A_j(\Gamma)\coloneqq\left\lbrace\alpha\in\Gamma|\mbox{ for all open neighborhoods  }W\mbox{ of }\alpha,\restr{\pi_j}{W\cap\Gamma}\mbox{ is not injective}\right\rbrace,$$
for $j=1,2$ and $B_j(\Gamma)\coloneqq\pi_j\left(A_j(\Gamma)\right).$

We extend the definition to holomorphic 1-chains $\Gamma=\sum_i n_i\Gamma(i)$ by
$$A_j(\Gamma)\coloneqq\bigcup_iA_j(\Gamma(i)),\mbox{ and }B_j(\Gamma)\coloneqq\pi_j(A_j(\Gamma)).$$
We call $A_2(\Gamma)$ the set of \emph{ramification points} of the holomorphic correspondence associated to $\Gamma$, and $B_2(\Gamma)$ the set of its \emph{critical values}.
\end{definition}

Note that $A_1(\Gamma)=\iota(A_2(\Gamma^{-1}))$ and $A_2(\Gamma)=\iota(A_1(\Gamma^{-1}))$, where $\iota:\widehat{\mathbb{C}}\times\widehat{\mathbb{C}}\to\widehat{\mathbb{C}}\times\widehat{\mathbb{C}}$ is the involution $\iota(z,w)=(w,z)$.

Suppose $\Gamma$ is the graph of a holomorphic correspondence on $\widehat{\mathbb{C}}$. Let $g:\Omega\to\widehat{\mathbb{C}}$ be a holomorphic function defined on a domain $\Omega$ of $\widehat{\mathbb{C}}$, whose graph $\graph(g)$ is contained in one of the irreducible components $\Gamma(i)$ of $\Gamma$. If $a\in\Omega$ is a critical point for $g$, then $(a,g(a))\in A_2(\Gamma(i))\subset A_2(\Gamma)$ and therefore $g(a)\in B_2(\Gamma)$, \emph{i.e.}, the critical value $\restr{\pi_2}{\Gamma(i)}(\alpha)=g(a)$ of the function $g$ is a critical value for the holomorphic correspondence associated to $\Gamma$, as well.

On the other hand, if $\alpha\notin A_1(\Gamma)$, then there exists a holomorphic function $g:\Omega\to\widehat{\mathbb{C}}$ defined on a neighborhood $\Omega$ of $a=\pi(\alpha)$, such that $(a,g(a))=\alpha$, and $\graph(g)\subset\Gamma(i)$ for some $i$. If in addition $\alpha\in A_2(\Gamma)$, then $g$ is not locally injective at $a$, and therefore $g'(a)=0$. Therefore $a$ is a critical point of $g$ and $\pi_2(\alpha)=g(a)$ is a critical value of $g$.

If we denote by $\critpt(g)$ the set of critical points of $g$, and by
$$\critval(g)\coloneqq\lbrace g(a)|a\in\critpt(g)\rbrace$$
the set of critical values of $g$, then we get a motivation for the name ``critical values" in Definition \ref{crit} by the containment
$$B_2(\Gamma)\setminus B_1(\Gamma)\subset\bigcup\limits_{i}\bigcup\limits_{\graph(g)\subset\Gamma(i)}\critval(g),$$
where the first union runs over the irreducible components of $\Gamma$ and the second union runs over all the holomorphic functions $g:\Omega\to\widehat{\mathbb{C}}$ whose graph $\graph(g)$ is contained in $\Gamma(i)$.

\begin{proposition}\label{b2} For every $a\in\mathbb{C}\setminus\lbrace 1\rbrace$, we have that
\begin{equation}\label{a1}
A_1(\Gamma_a)=\left\lbrace \left(\infty,\frac{a+1}{2}\right),\left( -2,1\right),\left(2,\frac{3a+1}{3+a}\right)\right\rbrace,
\end{equation}
and 
\begin{equation}\label{a2}
A_2(\Gamma_a)=\left\lbrace \left(\infty,\frac{a+1}{2}\right),\left(1,\frac{4a+2}{a+5}\right), \left(-1,\frac{2}{3-a}\right)\right\rbrace.
\end{equation}
As a consequence, $B_1(\Gamma_a)=\lbrace\infty,-2,2\rbrace$ and $B_2(\Gamma_a)=\left\lbrace \frac{a+1}{2},\frac{4a+2}{a+5},\frac{2}{3-a}\right\rbrace$. Moreover, $\Gamma_a$ is smooth at all points except $\left(\infty,\frac{a+1}{2}\right)$.
\end{proposition}
To prove this proposition, we first prove the following lemma.
\begin{lemma}\label{cov lemma} For each $a\in\mathbb{C}\setminus\lbrace 1\rbrace$,
$$A_1(\Gamma_{\cov})=\lbrace(\infty,\infty),(-2,1),(2,-1)\rbrace$$
and
$$A_2(\Gamma_{\cov})=\lbrace(\infty,\infty),(1,-2),(-1,2)\rbrace.$$
Thus, $B_1(\Gamma_{\cov})=\lbrace\infty, -2,2\rbrace$ and $B_2(\Gamma_{\cov})=\lbrace\infty, 1,-1\rbrace$.
\end{lemma}
\begin{proof}
Differentiating the equation $P_Q(z,w)=0$ with respect to $w$, we get that $\partial_w P_Q(z,w)=z+2w$ vanishes if and only if $w=\frac{-z}{2}$. Also, $P_Q(z,-\frac{z}{2})=0$ if and only if $z=\pm 2$. Therefore, $\frac{dw}{dz}$ exists on $\mathbb{C}\setminus\lbrace -2, 2\rbrace$. Thus, by the Implicit Function Theorem, for every $(z,w)\in\Gamma_{\cov}$ such that $z\in\mathbb{C}\setminus\lbrace-2,2\rbrace$, there exists a domain $\Omega$ containing $z$ and a holomorphic function $g:\Omega\to\widehat{\mathbb{C}}$ such that $g(z)=w$ and  $\graph(g)=\Gamma_{\cov}\cap U$, for some open neighborhood $U$ of $(z,w)$. In addition, the function $g$ will be locally injective at $z$ if $\partial_z P_Q(z,w)=2z+w$ is nonzero. Therefore, if $(z,w)\in\Gamma_{\cov}$ satisfies that both $z$ and $w$ are different from $\pm 2$, then $(z,w)\notin A_2(\Gamma_{\cov})$.

On the other hand, observe that the only points $(z,w)\in\Gamma_{\cov}$ with $z=\pm 2$ are $(-2,1)$ and $(2,-1)$. In particular, $w\neq\pm 2$, and by the symmetry of the above argument, we can use the Implicit Function Theorem to obtain a neighborhood $U$ of $w$ and a function $g:U\to\widehat{\mathbb{C}}$ satisfying $\graph(g)\subset\Gamma_{\cov}$ and $g(w)=z$, and such that $\restr{\pi_2}{\Gamma_{\cov}}$ is injective in the neighborhood $(g(U)\times U)\cap\Gamma_{\cov}$ of $(z,w)$. This proves that neither $(-2,1)$ or $(2,-1)$ belong to $A_2(\Gamma_{\cov})$. Since the only points $(z,w)\in\Gamma_{\cov}$ with $z=\pm 1$ are $(1,-2)$ and $(-1,2)$, and since $\Gamma_{\cov}\setminus(\mathbb{C}\times\mathbb{C})=\lbrace (\infty,\infty)\rbrace$, we have that $A_2(\Gamma_{\cov})$ is contained in $\lbrace (\infty,\infty),(1,-2),(-1,2)\rbrace$.

We will check that for every neighborhood $W$ of $(\infty,\infty)$, $(1,-2)$ and $(-1,2)$, we have that $\restr{\pi_2}{W\cap\Gamma_{\cov}}$ is not injective. Let $W_1,W_2$ and $W_3$ be open neighborhoods of $(\infty,\infty)$, $(1,-2)$ and $(-1,2)$, respectively. Then there exists $T>0$ such that  for every $0<t<T$,
$$\left(\frac{1}{2}\left(\pm\sqrt{3}\sqrt{-\frac{1}{t^2}-\frac{4}{t}}+\frac{1}{t}+2\right),-2-\frac{1}{t}\right)\in\Gamma_{\cov}\cap W_1,$$
$$\left(\frac{1}{2}\left(\pm\sqrt{3}\sqrt{-t^2-4t}+t+2\right),-2-t\right)\in\Gamma_{\cov}\cap W_2,$$
and
$$\left(\frac{1}{2}\left(\pm\sqrt{3}\sqrt{4t-t^2}+t-2\right),2-t\right)\in\Gamma_{\cov}\cap W_3.$$We conclude that $A_2(\Gamma_{\cov})=\lbrace (\infty,\infty),(1,-2),(-1,2)\rbrace$, and by the symmetry of $\cov$, $A_1(\Gamma_{\cov})=\lbrace (\infty,\infty),(-2,1),(2,-1)\rbrace$. We conclude that
$$B_2(\Gamma_{\cov})=\lbrace \infty, -2,2\rbrace=B_1(\Gamma_{\cov}).$$
\end{proof}
\begin{proof}[Proof of Proposition \ref{b2}]
Since $\F=\J\circ\cov$ and $\J$ is an involution, we have that 
$$A_j(\Gamma_a)=\lbrace (z,\J(w)):(z,w)\in A_j(\Gamma_{\cov})\rbrace,$$
for $j=1,2$. Using Lemma \ref{cov lemma}, this gives us (\ref{a1}) and (\ref{a2}), and thus $B_1(\Gamma_a)=\lbrace \infty, -2,2\rbrace$ and
$$B_2(\Gamma_a)=\lbrace \J(\infty),\J(-2),\J(2)\rbrace=\left\lbrace \frac{a+1}{2},\frac{4a+2}{a+5},\frac{2}{3-a}\right\rbrace.$$

In addition, observe that locally, $\Gamma_a$ is either a function on $z$ or on $w$ for all $(z,w)\in\Gamma_a$ with $z\neq\infty$, then the only point that can be irregular is $\left(\infty,\frac{a+1}{2}\right)$. Indeed, this point is irregular, as the curve given by the points 
$$\left(\frac{1}{2}\left(\pm\sqrt{3}\sqrt{-\frac{1}{t^2}-\frac{4}{t}}+\frac{1}{t}+2\right),-2-\frac{1}{t}\right)\in\Gamma_{\cov}$$
self-intersects at $(\infty,\infty)$ with an angle of $\frac{2\pi}{3}$. In other words, there are two functions $z(w)$ which intersect with different derivatives, which makes $(\infty,\infty)$ an irregular point of $\Gamma_{\cov}$. Thus, passing through the involution $\J$, we get that $\left(\infty,\frac{a+1}{2}\right)$ is an irregular point of  $\Gamma_a$.
\end{proof}

\begin{remark}\label{open} The correspondence $\cov$, and hence $\F$, sends open sets to open sets. Indeed, let $U\subset\widehat{\mathbb{C}}$ be open and take $w_0\in\cov(U)$. Then there exists $z_0\in U$ for which $(z_0,w_0)\in\Gamma_{\cov}$. We will prove that $\cov(U)$ is open by showing that in all the cases, $w_0\in\interior(\cov(U))$.
\begin{itemize}
\item Suppose $(z_0,w_0)\notin A_1(\Gamma_{\cov})$. By Lemma \ref{cov lemma} and since {$(\cov)^{-1}(\infty)=\lbrace\infty\rbrace$}, we have that $w_0\neq\infty$. Moreover, there exists a holomorphic function $g:\Omega\to\mathbb{C}$ on an open subset $\Omega\subset U$, and $(z_0,w_0)\in\graph(g)\subset\Gamma_{\cov}$. Furthermore, $\Gamma_{\cov}$ is irreducible and $\cov$ is not constant, so $g$ is not constant. Thus, $g$ is open and then $w_0\in g(\Omega)\subset\interior(\cov(U))$.
\item Now suppose $(z_0,w_0)\notin A_2(\Gamma_{\cov})$. By Lemma \ref{cov lemma} and since $\cov(\infty)=\lbrace\infty\rbrace$, then $z_0\neq\infty$ and there exists a holomorphic function $\tilde{g}:\widetilde{\Omega}\to\mathbb{C}$ on an open subset $\widetilde{\Omega}\subset\mathbb{C}$ so that $(z_0,w_0)\in\iota(\graph(\tilde{g}))\subset\Gamma_{\cov}$. Since $\tilde{g}$ is continuous, then $\tilde{g}^{-1}(U)$ is open and $w_0\in\tilde{g}^{-1}(U)\subset\interior(\cov(U))$.
\item Finally, if $(z_0,w_0)\in A_1(\Gamma_{\cov})\cap A_2(\Gamma_{\cov})$, then Lemma \ref{cov lemma} implies that $(z_0,w_0)=(\infty,\infty)$. For each $r>0$, put $U_r\coloneqq\lbrace |z|>r\rbrace\cup\lbrace\infty\rbrace$. To show that $\infty\in int(\cov(U))$ we will show that for every $R>\sqrt{3}$, $U_{2R}\subset\cov(U_R)$. We proceed by contrapositive. If $P_Q(z,w)=0$ and $|w|\leq R$, then
$$|z|^2=|3-zw-w^2|\leq 3+|z||w|+|w|^2\leq 2R^2+|z|R.$$
Hence $|z|^2-R|z|-2R^2\leq 0$ and $|z|\leq 2R$. This implies that for every $R>\sqrt{3}$, we have that $\cov(U_{2R})\subset U_R$. By the symmetry of $\cov$,
$$U_{2R}\subset\cov(\cov(U_{2R}))\subset\cov(U_R).$$
Moreover, if $R>\sqrt{3}$ is large enough so that $U_R\subset U$, then 
$$\infty\in U_{2R}\subset \interior(\cov(U)).$$
\end{itemize}
Therefore, $\cov$ sends open sets to open sets. Since the involution $\J$ also sends open sets to open sets, then so does $\F$.
\end{remark}

%% file: Klein.tex
In this section, we will define the set $\mathcal{K}$ of parameters we will consider for our family. 

\begin{definition} A \emph{Fundamental Domain} for an irreducible holomorphic correspondence $F$ is an open set $\Delta_F$ that is maximal with the property that $\Delta_F\cap F(\Delta_F)=\emptyset$.
\end{definition}

\begin{definition} We say that a pair $(\Delta_{\cov},\Delta_{\J})$ of fundamental domains for $\cov$ and $\J$, respectively, is a \emph{Klein combination pair} for $\F$ if both $\Delta_{\cov}$ and $\Delta_{\J}$ are simply connected domains, bounded by Jordan curves, and satisfy
$$\Delta_{\cov}\cup\Delta_{\J}=\widehat{\mathbb{C}}\setminus\lbrace 1\rbrace.$$
We define as well the \emph{Klein combination locus} $\mathcal{K}$ to be the set of parameters $a\in\mathbb{C}\setminus\lbrace 1\rbrace$ for which there exist a Klein combination pair.
\end{definition}

In \cite{BulLom20I}, for $|a-4|\leq 3$, $a\neq 1$, the authors found a Klein combination pair for $\F$, where $\Delta_{\cov}$ is given by the right side of the curve 
$$L\coloneqq\left\lbrace\left(1+\frac{t}{2}\right)\pm i\sqrt{3\left(t+\frac{t}{2}^2\right)}\hspace{.1cm}\Big|t\in[0,\infty]\right\rbrace=\cov((-\infty ,-2]),$$
and $\Delta_{\J}$ is given by the exterior of the circle passing through $z=1$ and $z=a$ with diameter contained in the real line.

\begin{figure}[h]
\centering
\includegraphics[scale=.8]{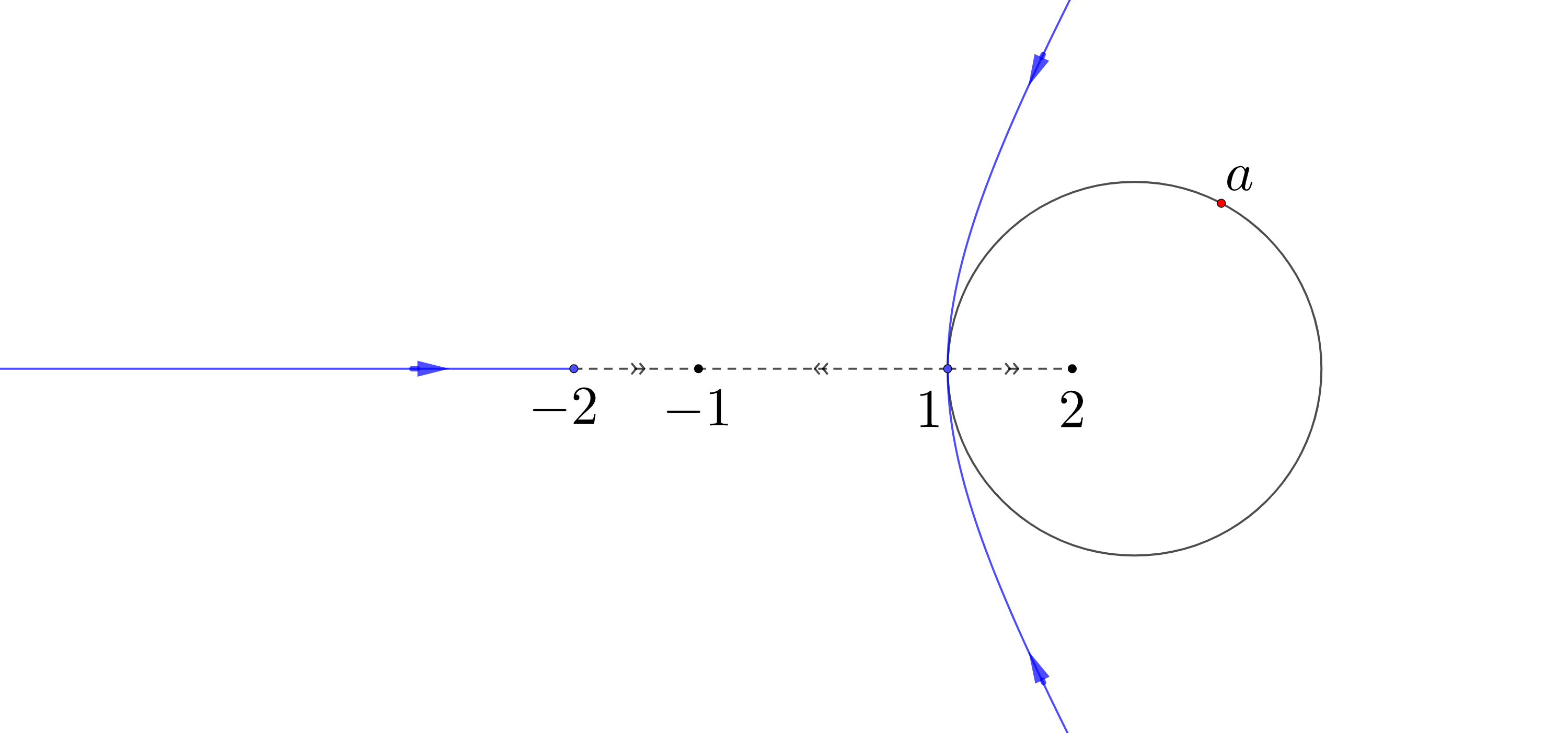}
\caption{Klein Combination pair for $|a-4|\leq 3$.}\label{klein}
\end{figure}

This pair $(\Delta_{\cov},\Delta_{\J})$ is composed by simply connected fundamental domains for $\cov$ and $\J$, respectively, whose boundaries are Jordan curves, smooth except from $\partial\Delta_{\cov}$ at $(\infty,\infty)$ (see Figure \ref{klein} and \cite[Proposition 3.3]{BulLom20I}). In particular,
$$\lbrace a\in\mathbb{C}|\mbox{ } |a-4|\leq 3\rbrace\setminus\lbrace 1\rbrace\subset\mathcal{K}.$$

From now on, whenever $a\in\mathcal{K}$, we denote by $(\Delta_{\cov},\Delta_{\J})$ a Klein combination pair for $\F$.

The following remark will be useful to prove that $\F$ is not modular, and later to analyze the asymptotic behavior of $\F$.

\begin{remark}\label{containment} Let $a\in\mathcal{K}$.
\begin{itemize}
\item[1.] Note that $\F(\widehat{\mathbb{C}}\setminus\Delta_{\J})\subset\widehat{\mathbb{C}}\setminus\Delta_{\J}$ and ${\F((\widehat{\mathbb{C}}\setminus\Delta_{\J})\setminus\lbrace 1\rbrace)\subset\widehat{\mathbb{C}}\setminus\overline{\Delta_{\J}}}$. Indeed, 
$$\F(1)=\J(\cov (1))=\J(\lbrace 1,-2\rbrace)=\left\lbrace 1,\frac{4a+2}{a+5}\right\rbrace.$$
From Remark \ref{open}, we have that $\cov$ sends open sets to open sets. In particular, $-2$ cannot belong to $\Delta_{\cov}$, as $1\notin\cov(\Delta_{\cov})=\widehat{\mathbb{C}}\setminus\overline{\Delta_{\cov}}$. By the Klein combination pair condition, this implies that $-2\in\Delta_{\J}$, and thus $\J(-2)=\frac{4a+2}{a+5}\notin\Delta_{\J}$. Since we also have that $1\notin\Delta_{\J}$, we have that
$$\F(1)\subset\widehat{\mathbb{C}}\setminus\Delta_{\J}.$$
On the other hand, for $z\in\widehat{\mathbb{C}}\setminus\Delta_{\J}$, $z\neq 1$, we have that $z\in\Delta_{\cov}$. Therefore, 
$$\cov(z)\subset\widehat{\mathbb{C}}\setminus\overline{\Delta_{\cov}}\subset\Delta_{\J},$$
and thus, since $\J$ is an involution fixing $\partial\Delta_{\J}$,
$$\F(z)=\J\circ\cov(z)\subset\J(\Delta_{\J})=\widehat{\mathbb{C}}\setminus\overline{\Delta_{\J}}.$$
In particular, $\Delta_{\J}\setminus\F(\widehat{\mathbb{C}}\setminus\Delta_{\J})\neq\emptyset$.
\item[2.] Observe that $\F^{-1}(\Delta_{\J})\subset\Delta_{\J}$ and $\F^{-1}(\overline{\Delta_{\J}})\subset\Delta_{\J}\cup\lbrace 1\rbrace$. Indeed, for every $z\in\Delta_{\J}$, we have that $\J(z)\in\widehat{\mathbb{C}}\setminus\overline{\Delta_{\J}}\subset\Delta_{\cov}$. From Remark \ref{open} and since $(\Delta'_{\cov},\Delta_{\J})$ is a Klein combination pair, then 
$$\F^{-1}(z)=\cov(\widehat{\mathbb{C}}\setminus\overline{\Delta_{\J}})\subset\cov(\Delta_{\cov})\widehat{\mathbb{C}}\setminus\overline{\Delta_{\cov}}\subset\Delta_{\J}.$$
Furthermore, for $w\in\overline{\Delta_{\J}}$, we have that
$$\F^{-1}(w)\subset\cov(\J(\overline{\Delta_{\J}}))\subset\cov(\widehat{\mathbb{C}}\setminus\Delta_{\J})\subset\cov(\overline{\Delta_{\cov}})\subset\widehat{\mathbb{C}}\setminus\Delta_{\cov}\subset\overline{\Delta_{\J}}.$$
\end{itemize}
\end{remark}

%% file: WM.tex
In this section we prove that $\F$ is weakly-modular but not modular, for every $a\in\mathcal{K}$. In addition, we prove that it does not satisfy the required conditions for the equidistribution result \cite[Theorem 3.5]{bhasri16}. All together prove Theorem \ref{A}.

\subsection{Modularity and weak modularity}\label{modularity}

\begin{definition}\label{modular} Let $G$ be a connected Lie Group, $\Lambda$ a torsion free lattice and $K$ a compact Lie subgroup. Let $g\in G$ be such that $g\Lambda g^{-1}\cap\Lambda$ has finite index in $\Lambda$. The \emph{irreducible modular correspondence} induced by $g$ is the multivalued map $F_g$ on $X=\Lambda\setminus G/K$ corresponding to the projection to $X$ of the map $x\mapsto (x,gx)$ on $G\to G\times G$. Denote by $\Gamma_g$ the graph of $F_g$.
A \emph{modular correspondence} $F$ is a correspondence whose graph is of the form $\sum_jn_j\Gamma_{g_j}$, for $\Gamma_{g_j}$ as before.
\end{definition}
The following definition was introduced by Dinh, Kaufmann and Wu in \cite{DinKauWu20}.

\begin{definition}\label{weaklymodular} Let $X$ be a compact Riemann surface and let $F$ be a holomorphic correspondence on $X$ with graph $\Gamma$ such that $d(F)=d(F^{-1})$. We say that $F$ is a \emph{weakly-modular correspondence} if there exist Borel probability measures $\mu_1$ and $\mu_2$ on $X$, such that
$$(\restr{\pi_1}{\Gamma})^*\mu_1=(\restr{\pi_2}{\Gamma})^*\mu_2.$$
\end{definition}

\begin{remark}\label{prop modular} \hfill
\begin{itemize}
\item[1.] Let $F$ be a modular correspondence that is also a holomorphic correspondence. Then it is always the case that $$d(F)=\sum_j n_j[\Lambda: g_j\Lambda g_j^{-1}\cap\Lambda]=\sum_j n_j[\Lambda:g_j^{-1}\Lambda g_j\cap\Lambda]=d(F^{-1}).$$
\item[2.] With the notation in Definition \ref{modular}, let $\lambda$ be the direct image on $X$ of the finite Haar measure on $\Lambda\setminus G$. Then $\frac{1}{d(F_g)}F_g^*\lambda=\lambda$ and if we put $\mu_1=\mu_2=\lambda$, we get that
$$(\restr{\pi_1}{\Gamma_g})^*\mu_1=(\restr{\pi_2}{\Gamma_g})^*\mu_2.$$
Therefore, modular correspondences are weakly modular.
\item[3.] The measure $\lambda$ above is Borel, invariant under $F_g$ and assigns positive measure to nonempty open sets.
\end{itemize}
\end{remark}
\vspace{.1cm}

\begin{proof}[Proof of Theorem \ref{A} part 1.]
Observe that the graph $\Gamma_{\cov}$ of the correspondence $\cov$ is symmetric with respect to the diagonal $\mathfrak{D}_{\widehat{\mathbb{C}}}=\lbrace (z,z) | z\in\widehat{\mathbb{C}}\rbrace$, as $z\in\cov(w)$ if and only if $w\in\cov(z)$. Let $m$ be any positive and finite measure on $\Gamma_{\cov}$, and $\iota:\Gamma_{\cov}\to\Gamma_{\cov}$ the involution $\iota(z,w)\coloneqq(w,z)$. Take 
$$m_0\coloneqq\left(\restr{\pi_1}{\Gamma_{\cov}}\right)^*\left(\restr{\pi_1}{\Gamma_{\cov}}\right)_*m+\iota^*\left(\restr{\pi_1}{\Gamma_{\cov}}\right)^*\left(\restr{\pi_1}{\Gamma_{\cov}}\right)_*m.$$
The measure $m_0$ is symmetric in $\Gamma_{\cov}$ in the sense that $\iota^*m'=m'$ and moreover
\begin{equation}\label{zeroth}\left(\restr{\pi_1}{\Gamma_{\cov}}\right)^*\left(\restr{\pi_1}{\Gamma_{\cov}}\right)_*m_0=\left(\restr{\pi_2}{\Gamma_{\cov}}\right)^*\left(\restr{\pi_2}{\Gamma_{\cov}}\right)_*m_0.\end{equation}
After normalizing if necessary, this proves that $\cov$ is weakly modular with measures $\mu'_1\coloneqq\left(\restr{\pi_1}{\Gamma_{\cov}}\right)_*m_0$ and $\mu'_2\coloneqq\left(\restr{\pi_2}{\Gamma_{\cov}}\right)_*m_0$. Our goal is to show there are probability measures $\mu_1$ and $\mu_2$ on $\widehat{\mathbb{C}}$ such that
\begin{equation}\label{wm}
\left(\restr{\pi_1}{\Gamma_a}\right)^*\mu_1=\left(\restr{\pi_2}{\Gamma_a}\right)^*\mu_2,
\end{equation}
where $\Gamma_a$ is the graph of the correspondence $\F=\F\circ\cov$.
\\
Put $\mu_2\coloneqq(\J)^*\mu'_2=\left(\J\right)_*\mu'_2$ and observe that by symmetry of $\J$ we have that
\begin{equation}\label{first}
\left(\restr{\pi_2}{\Gamma_a}\right)^*\mu_2=\left(\restr{\pi_2}{\Gamma_a}\right)^*\left(\J\right)^*\mu'_2=\left(\J\circ\restr{\pi_2}{\Gamma_a}\right)^*\mu'_2
\end{equation} 
Now let $\widehat{\J}:\widehat{\mathbb{C}}\times\widehat{\mathbb{C}}\to\widehat{\mathbb{C}}\times\widehat{\mathbb{C}}$ be given by $\widehat{\J}(z,w)\coloneqq (z,\J(w))$. Observe that whenever $(z,w)\in\Gamma_{\cov}$, then $(z,\J(w))\in\Gamma_a$ and 
\begin{equation}\label{second}
\left(\J\circ\restr{\pi_2}{\Gamma_{\cov}}\right)(z,w)=\J(w)=\left(\restr{\pi_2}{\Gamma_a}\circ\widehat{\J}\right)(z,w).
\end{equation}
On the other hand, $\restr{\pi_1}{\Gamma_{\cov}}\circ\widehat{\J}(z,w)=\restr{\pi_1}{\Gamma_a}(z,w)$ with multiplicity $2$. This, together with equations (\ref{zeroth}), (\ref{first}) and (\ref{second}) yield
\begin{equation}\label{w-m}
\left(\restr{\pi_2}{\Gamma_a}\right)^*\mu_2=\left(\widehat{\J}\right)^*\left(\restr{\pi_2}{\Gamma_{\cov}}\right)^*\mu'_2=\left(\widehat{\J}\right)^*\left(\restr{\pi_1}{\Gamma_{\cov}}\right)^*\mu'_1=2\left(\restr{\pi_1}{\Gamma_{\cov}}\right)^*\mu'_1.
\end{equation}
Observe that $\mu_1\coloneqq 2\mu'_1$ and $\mu_2$ both have $2$ times the mass of $\mu_1$ and $\mu_2$. After normalizing, equation (\ref{w-m}) proves that $\F$ is weakly-modular, as desired.

In order to check $\F$ is not modular, we will prove that no Borel measure $\lambda$ on $\widehat{\mathbb{C}}$ that gives positive measure to nonempty open sets can be invariant under $\F$. Suppose by contradiction that $\lambda$ is such a measure satisfying $\frac{1}{2}\F^*\lambda=\lambda$. In particular, we have that $\lambda\left((\widehat{\mathbb{C}}\setminus\Delta_{\J})\setminus\F(\widehat{\mathbb{C}}\setminus\Delta_{J})\right)=0$. On the other hand, note that
$$\F(\widehat{\mathbb{C}}\setminus\Delta_{\J})=\restr{\pi_2}{\Gamma_a}\left(\Gamma_a\cap ((\widehat{\mathbb{C}}\setminus\Delta_{\J})\times\widehat{\mathbb{C}})\right)$$
is closed in $\widehat{\mathbb{C}}$. By Remark \ref{containment} part 1, $(\widehat{\mathbb{C}}\setminus\overline{\Delta_{\J}})\setminus\F(\widehat{\mathbb{C}}\setminus\Delta_{\J})$ is open and nonempty, contained in $(\widehat{\mathbb{C}}\setminus\Delta_{\J})\setminus\F(\widehat{\mathbb{C}}\setminus\Delta_{J})$. This contradicts part $3$ of Remark \ref{prop modular}, as $\lambda$ cannot assign 0 measure to open sets.
\end{proof}

\begin{remark}\label{norm} For a $(d,d)$ holomorphic correspondence $F$ on compact Riemann surface $X$, the operator $\frac{1}{d}F^*$ acts on the space $L^2_{(1,0)}$ of $(1,0)$-forms with $L^2$ coefficients. In \cite{DinKauWu20} the authors showed that the operator norm satisfies $\Vert\frac{1}{d}F^*\Vert\leq 1$, with strict inequality for non weakly-modular correspondences. This strict inequality is a key factor of their equidistribution result. However, this is never the case for $F=\F$, $a\neq 1$. We claim that
$$\Big\Vert\frac{1}{2}\F^*\Big\Vert=\sup\left\lbrace\Big\Vert\frac{1}{2}\F^*\phi\Big\Vert_{L^2}\Big|\phi\in L^2_{(1,0)},\Vert\phi\Vert_{L^2}=1\right\rbrace=1,$$
and furthermore that the supremum is attained. In order to prove this, we use that $\|\frac{1}{d}F^*\phi\|_{L^2}=\|\phi\|_{L^2}$ for $\phi\in L^2_{(1,0)}$ if and only if for every $U\subset X\setminus B_1(\Gamma)$ and for every pair of local branches $f_1$ and $f_2$ of $F$ on $U$, the equality $f_1^*\phi=f_2^*\phi$ holds on $U$ (see \cite[Proposition 2.1]{DinKauWu20}).
\\
Observe that the form $\phi(z)=e^{-|Q(z)|}dz$ belongs to $L^2_{(1,0)}$ for $Q$ as in Section \ref{thefamily}. Let $U\subset \widehat{\mathbb{C}}\setminus B_1(\Gamma_{\cov})$. Then the deleted covering correspondence $\cov$ sends $z$ to the values $w$ for which $\frac{Q(z)-Q(w)}{z-w}=0$. Hence, any two local branches $f_1$ and $f_2$ of $\cov$ satisfy $f_1^*\phi(z)=Q(z)=f_2^*\phi(z)$, and thus $\Vert\frac{1}{2}\cov^*\Vert=1$. Now note that $\J$ is an involution, and hence $\J^*$ has operator norm $1$. Thus we can conclude that $\Vert\frac{1}{2}\F^*\Vert=1$ as well, with supremum attained at $\phi$.
\end{remark}

\subsection{Limit Sets}\label{limitsets}

In this section, we define limit sets and give some properties. We will also prove that the application listed in \cite[Section 7]{bhasri16} does not hold for our family of correspondences, and hence this is a new case to study equidistribution.

\begin{remark}\label{taylor} From Proposition \ref{b2}, observe that $1\notin B_1(\Gamma_a)$, and hence there is a holomorphic function $g$ whose graph contains $(1,1)$ and is contained in $\Gamma_a$. After the change of coordinates $\psi(z)=z-1$, the function $g$ has Taylor series expansion
$$g^{\psi}(z)=z+\frac{a-7}{3(a-1)}z^2+\cdots$$
whenever $a\neq 7$, and
$$g^{\psi}(z)=z+\frac{1}{27}z^4+\cdots$$
for $a=7$ (see \cite[Proposition 3.5]{BulLom20I}). In particular, $g^{\psi}$ has a parabolic fixed point at $0$ with multiplier $1$.
\end{remark}

In \cite[Proposition 3.8]{BulLom20I}, the authors showed that for each $a\in\mathcal{K}$, after a small perturbation of $\partial\Delta_{\cov}$ and $\partial\Delta_{\J}$ around $z=1$, we can choose a Klein combination pair $(\Delta'_{\cov},\Delta'_{\J})$ so that $\partial\Delta'_{\cov}$ and $\partial\Delta'_{\J}$ are both smooth at $1$, and transverse at $1$ to the line generated by the repelling direction at $z=1$ for $a\neq 7$, and to the real axis in the case $a=7$.

For $a\in\mathcal{K}$ and $(\Delta_{\cov},\Delta_{\J})$ as above, we define 
$$\Lambda_{a,+}\coloneqq\bigcap\limits_{n=0}^\infty \F^n(\widehat{\mathbb{C}}\setminus\Delta'_{\J})\hspace{.5cm}\mbox{ and }\hspace{.5cm}\Lambda_{a,-}\coloneqq\bigcap\limits_{n=0}^\infty \F^{-n}(\overline{\Delta'_{\J}})$$
to be the \emph{forward and backward limit set of $\F$}, respectively. These sets do not depend on the choice of the Klein combination pair $(\Delta'_{\cov},\Delta'_{\J})$ as above.

\begin{lemma}\label{lambdas} Let $a\in\mathcal{K}$. We have that
\begin{itemize}
\item[1.] $\J(\Lambda_{a,\pm})=\Lambda_{a,\mp}$ and $\J(\partial\Lambda_{a,\pm})=\partial\Lambda_{a,\mp}$,
\item[2.] $\Lambda_{a,-}\cap\Lambda_{a,+}=\lbrace 1\rbrace,$
\item[3.] if $z\notin\Lambda_{a,-}$, then there exists $n\geq 1$ so that $\F^{n}(z)\subset\widehat{\mathbb{C}}\setminus\Delta'_{\J}$, and if $z\notin\Lambda_{a,+}$, then there exists $n\geq 1$ such that $\F^{-n}(z)\subset\Delta'_{\J}$.
\item[4.] $\F^{-1}(\Lambda_{a,-})=\Lambda_{a,-}$ and $\F^{-1}(\partial\Lambda_{a,-})=\partial\Lambda_{a,-}$, and
\item[5.] $\F(\Lambda_{a,+})=\Lambda_{a,+}$ and $\F(\partial\Lambda_{a,+})=\partial\Lambda_{a,+}$.
\end{itemize}
\end{lemma}

\begin{proof}
Since $\J$ is an involution sending $\overline{\Delta'_{\J}}$ and $\widehat{\mathbb{C}}\setminus\Delta'_{\J}$ to each other, and since $\J\circ\F^n=\F^{-n}\circ\J$, then
$$\J(\Lambda_{a,+})=\bigcap\limits_{n=0}^\infty\J\circ\F^n(\widehat{\mathbb{C}}\setminus\Delta'_{\J})=\bigcap\limits_{n=0}^\infty\F^{-n}\circ\J(\widehat{\mathbb{C}}\setminus\Delta'_{\J})=\Lambda_{a,-},$$
and applying $\J$ to both sides we also get $\J(\Lambda_{a,-})=\Lambda_{a,+}$. Moreover, since $\J$ is continuous,

Note that 
$$\Lambda_{a,-}\cap\Lambda_{a,+}\subset\overline{\Delta'_{\J}}\cap(\widehat{\mathbb{C}}\setminus\Delta'_{\J})=\partial\Delta'_{\J}.$$
Since $(\Delta'_{\cov},\Delta'_{\J})$ is a Klein combination pair, $\widehat{\mathbb{C}}\setminus\Delta'_{\J}\subset\Delta'_{\cov}$, so $\partial\Delta'_{\J}\subset\overline{\Delta'_{\cov}}$.

If $z\in\Delta'_{\cov}$, then $\cov(z)\subset\widehat{\mathbb{C}}\setminus\Delta'_{\J}$, and $\F(z)=\J\circ\cov(z)\subset\Delta'_{\J}$. This is a contradiction, as $\F(z)$ must belong to $\Lambda_{a,+}\subset\widehat{\mathbb{C}}\setminus\Delta'_{\J}$ as well. Again, since $(\Delta'_{\cov},\Delta'_{\J})$ is a Klein combination pair, we have that $z\in\partial\Delta'_{\cov}\cap\partial\Delta'_{\J}=\lbrace 1\rbrace$ and we conclude that
$$\Lambda_{a,-}\cap\Lambda_{a,+}=\lbrace 1\rbrace.$$

We prove part $3$ by contrapositive. Suppose that for all $n$, there exists $w\in\F^{n}(z)\cap(\widehat{\mathbb{C}}\setminus\Delta'_{\J})$, then $z\in\F^{n}(w)\subset\F^{n}(\widehat{\mathbb{C}}\setminus\Delta'_{\J})$ for all $n$. This implies that $z$ must belong to $\Lambda_{a,-}$. The other case is analogous.

It is immediate from the definition of $\Lambda_{a,-}$ and from part $2$ of Remark \ref{containment} that 
\begin{equation}\label{1st}
\F^{-1}(\Lambda_{a,-})=\Lambda_{a,-}.
\end{equation} 
From this and Remark \ref{open},
\begin{equation}\label{2nd}
\F^{-1}(\interior(\Lambda_{a,-}))\subset\interior(\F^{-1}(\Lambda_{a,-}))=\interior(\Lambda_{a,-}).
\end{equation}
Observe that if $z\in\partial\Lambda_{a,-}\setminus\lbrace 1\rbrace\subset\Delta'_{\cov}$, then
$$\F(z)=\J(\cov(z))\subset\J(\widehat{\mathbb{C}})\subset\J(\Delta'_{\J})\subset\widehat{\mathbb{C}}\setminus\overline{\Delta'_{\J}}.$$
Since $\Lambda_{a,-}\subset\overline{\Delta'_{\J}}$, then we conclude that
\begin{equation}\label{intersection}
\Lambda_{a,-}\cap\partial\Delta'_{\J}=\lbrace 1\rbrace
\end{equation}
Put $w\in\partial\Lambda_{a,-}$ and $z\in\F^{-1}(w)\subset\Lambda_{a,-}$. We will show that $z\in\partial\Lambda_{a,-}$, by contrapositive. Suppose $w\neq 1$ and $z\in\interior(\Lambda_{a,-})$. 
Then equation (\ref{intersection}) implies that
$$w\in\F(\interior(\Lambda_{a,-}))\cap\Delta'_{\J}.$$
From Remark \ref{open}, and the fact that $\Delta'_{\J}$ is open, we have that $\F(\interior(\Lambda_{a,-}))\cap\Delta'_{\J}$ is open. Moreover, for each $z'\in\interior(\Lambda_{a,-})$, the set $\F(z')$ consists of a point in $\Delta'_{\J}$ and one in $\widehat{\mathbb{C}}\setminus\Delta'_{\J}$. Since the one in $\Delta'_{\J}$ is actually in $\Lambda_{a,-}$ by definition of $\Lambda_{a,-}$, then $\F(\interior(\Lambda_{a,-}))\cap\Delta'_{\J}\subset\Lambda_{a,-}$. Then $w\in\interior(\Lambda_{a,-})$. This proves that
\begin{equation}\label{3rd}
\F^{-1}(\partial\Lambda_{a,-})\subset\Lambda_{a,-}.
\end{equation}
As for $w=1$, we have that $z\in\F^{-1}(1)=\lbrace -2,1\rbrace$. We know $1\in\partial\Lambda_{a,-}$, so it suffices to show that $-2\in\partial\Lambda_{a,-}$ as well. Let $U\subset\widehat{\mathbb{C}}$ be an open neighborhood of $-2$. From Remark \ref{open}, we have that $\F(U)$ is an open neighborhood of $1$. Since $1\in\partial\Delta'_{\J}$, and $\partial\Delta'_{\J}$ is a Jordan curve, then there exists a point $w'\in U\cap\partial\Delta'_{\J}$, $w'\neq 1$. We have that $\F(\partial\Delta'_{\J}\setminus\lbrace 1\rbrace)\subset\widehat{\mathbb{C}}\setminus\overline{\Delta'_{\cov}}$, so $\F^{-1}(w')$ consists of two points, $z',z''\in\widehat{\mathbb{C}}\setminus\overline{\Delta'_{\cov}}$, with $z'\in U$. We will prove that $z'\notin\Lambda_{a,-}$ by showing that $\F(z')\cap\Lambda_{a,-}=\emptyset$, and thus $-2\in\Lambda_{a,-}\setminus\interior(\Lambda_{a,-})=\partial\Lambda_{a,-}$. Indeed,
$$\F(z')=\J\left(\lbrace z'',\J(w')\rbrace\right)=\lbrace\J(z''),w'\rbrace\subset\widehat{\mathbb{C}}\setminus\Delta'_{\J}.$$
From part $2$ of Remark \ref{containment}, we have that $z'\notin\Lambda_{a,-}$, and $-2\in\partial\Lambda_{a,-}$ as desired. This, together with equations (\ref{1st}), (\ref{2nd}) and (\ref{3rd}) prove part $4$.

We have that part $1$ of Remark \ref{containment} together with the definition of $\Lambda_{a,+}$ prove that $\F(\Lambda_{a,+})=\Lambda_{a,+}$. To prove the rest of part $5$, we use part $4$ together with the fact that $\F=\J\circ\F^{-1}\circ\J$ and $\J(\partial\Lambda_{a,\pm})=\partial\Lambda_{a,\mp}$, as $\J$ is a continuous involution. Thus, 
$$\F(\partial\Lambda_{a,+})=\J\circ\F^{-1}\circ\J(\partial\Lambda_{a,+})=\J(\F^{-1}(\partial\Lambda_{a,-}))=\J(\partial\Lambda_{a,-})=\partial\Lambda_{a,+}.$$
\end{proof}

The following definition is from \cite{McG92} and \cite{bhasri16}.

\begin{definition}
Let $F$ be a holomorphic correspondence on $X$. We say that $\mathcal{R}\subset X$ is a \emph{repeller} for $F$ if there exists a set $U$ such that $\mathcal{R}$ is contained in the interior of $U$, and
$$\mathcal{R}=\bigcap_{K\in\mathfrak{K}(U,F^{-1})} K,$$
where
$$\mathfrak{K}(U,F^{-1})\coloneqq\lbrace K\subset X|F^{-1}(K)\subset K\mbox{ and }F^{-n}(U)\subset K\mbox{ for some }n\geq 0\rbrace.$$
\end{definition}
Bharali and Sridharan proved an equidistribution result similar to the one in this paper \cite[Theorem 3.5]{bhasri16} for correspondences having a repeller. Moreover, in \cite[Section 7.2]{bhasri16} they showed that there is a set of pairs $(a,k)$ for which their result can be applied to the correspondence
$$\left(\frac{az+1}{z+1}\right)^2+\left(\frac{az+1}{z+1}\right)\left(\frac{aw-1}{w-1}\right)+\left(\frac{aw-1}{w-1}\right)^2=3k.$$
For these correspondences (studied by Bullett and Harvey in \cite{BulHar00}), there is a set $\Lambda_{a,-}$ analogous to the one presented here (see \cite{BulPen01} for the general definition of limit sets). In order for the pair $(a,k)$ to work for their theorem, it is crucial for $\partial\Lambda_{a,-}$ to be a repeller for the correspondence. Nevertheless, part $2$ of Theorem \ref{A} says this never happens for $k=1$ and $|a-4|\leq 3$.

\begin{proof}[Proof of Theorem \ref{A} part 2.]

Let $U\subset\widehat{\mathbb{C}}$ contain $\partial\Lambda_{a,-}$ in its interior. We have that $1\in\partial\Lambda_{a,-}$ and is a parabolic fixed point of the function $g$ whose graph is contained in $\Gamma_a$, described in Remark \ref{taylor}. Take an attracting petal $\mathcal{P}$ at $1$ so that $\mathcal{P}\subset U$. We first show that every $K\in\mathfrak{K}(U,\F^{-1})$ contains $\partial\Lambda_{a,-}\cup\mathcal{P}$. Indeed, for every $K\in\mathfrak{K}(U,\F^{-1})$ and some integer $n\geq 0$,

\begin{equation}\label{petal}
\mathcal{P}\subset g^{-n}(\mathcal{P})\subset\F^{-n}(\mathcal{P})\subset\F^{-n}(U)\subset K.
\end{equation}

Moreover, from part 4 of  Lemma \ref{lambdas} we have that
\begin{equation}\label{limit}
\partial\Lambda_{a,-}\subset\F^{-n}(U)\subset K.
\end{equation}

Putting (\ref{petal}) and (\ref{limit}) together, we get that every $K\in\mathfrak{K}(U,\F^{-1})$ contains  the union $\partial\Lambda_{a,-}\cup\mathcal{P}$, and therefore so does the intersection over all $K$. Since $\interior(\partial\Lambda_{a,-})=\emptyset$ and $\interior(\mathcal{P})\neq\emptyset$, we have that
$$\partial\Lambda_{a,-}\subsetneq\partial\Lambda_{a,-}\cup\mathcal{P}\subset\bigcap_{K\in\mathfrak{K}(U,\F^{-1})} K.$$
Since $U$ is arbitrary, we conclude that $\partial\Lambda_{a,-}$ is not a repeller for $\F$.
\end{proof}

%% file: Exceptional.tex
In this section we will define a two-sided restriction of $\F$ and prove it is a proper holomorphic map of degree 2. We will find its exceptional set and the one of $\F$. This will be important for the next section, as it is the set of all points that may escape from the equidistribution property given in Theorem \ref{B}.

The following definition is classical.

\begin{definition} Let $f:U\to V$ be a holomorphic proper map, with $U,V$ open, $U\subset V$. For $z\in U$, we denote by $[ z]$ the equivalence class of $z$ by the equivalence relation
$$w\sim z\Leftrightarrow\exists n,m\in\mathbb{Z}^+\cup\lbrace 0\rbrace, f^n(w)=f^n(z).$$
We say that $z$ is \emph{exceptional} for $f$ if $[ z]$ is finite, and we call \emph{exceptional set} the set $\mathcal{E}$ of all points that are exceptional for $f$.
\end{definition}

Put $a\in\mathcal{K}$ and denote by $\f$ the two-sided restriction $\restr{\F}{}:\F^{-1}(\Delta'_{\J})\to \Delta'_{\J}$, meaning $\f$ sends each $z\in\F^{-1}(\Delta'_{\J})$ to the unique point in $\F(z)\cap \Delta'_{\J}$. This it is a single-valued, continuous and holomorphic, 2-to-1 map (see \cite[Proposition 3.4]{BulLom20I} and Theorem \ref{exc} below) that extends on a neighborhood of every point $z\in\partial\F^{-1}(\Delta_{\J})\setminus\lbrace -2\rbrace$. In particular, $\f$ extends around $z=1$ and $\f(1)=1$. Since $\Delta'_{\J}$ is open and $\f$ is continuous, then $\F^{-1}(\Delta'_{\J})(=\f^{-1}(\Delta'_{\J}))$ is open as well. Note that the analogous of part 2 of Remark \ref{containment} holds for $\Delta'_{\J}$ and $\Delta'_{\cov}$ instead, using the definition of Klein pair and the fact that $\J$ is open. Thus, $\F^{-1}(\Delta'_{\J})\subset\Delta'_{\J}$.
\vspace{.3cm}

The following theorem is the main result of this section.

\begin{theorem}\label{exc} For each $a\in\mathcal{K}$, we have the following.
\begin{itemize}
\item[1.] The two-sided restriction $\f:\F^{-1}(\Delta'_{\J})\to\Delta'_{\J}$ of $\F$ is holomorphic and proper, of degree $2$.
\item[2.] The map $\f$ has a critical point if and only if $2\in\widehat{\mathbb{C}}\setminus\overline{\Delta'_{\J}}$. Furthermore, in that case we have that the critical point is $-1$.
\item[3.] The exceptional set $\mathcal{E}_{a,-}$ of $\f$ is nonempty if and only if $a=5$. In that case, $\mathcal{E}_{a,-}=\lbrace -1\rbrace$.
\end{itemize}
\end{theorem}

Computing images and preimages under $\F$, we see that when $a=5$, we have that $\circlearrowright -1\mapsto 2\circlearrowleft$. In the following section, it will be useful to use the full orbit of $\mathcal{E}_{a,-}$ under $\F$, meaning
$$\E\coloneqq\left\lbrace\begin{matrix}
\emptyset & \mbox{ if }a\neq 5\\
\lbrace -1,2\rbrace & \mbox{ if }a=5.\end{matrix}\right.$$

\begin{proof}[Proof of Theorem \ref{exc}]
We first show part 1. Observe that $\infty$ and $\J(\infty)$ lie on opposite sides of $\partial\Delta_{\J}$. Therefore, $\left(\infty,\frac{a+1}{2}\right)\notin(\F^{-1}(\Delta'_{\J})\times\Delta'_{\J})$. In view of Proposition \ref{b2}, this implies that for every 
$$(z_0,w_0)\in\graph(\f)=(\F^{-1}(\Delta'_{\J})\times \Delta'_{\J})\cap\Gamma_a,$$
there there exists a neighborhood $U$ of $(z_0,w_0)$ such that $U\cap\Gamma_a$ is the graph of a holomorphic function either in $z$ or in $w$. Therefore $\graph(\f)$ is an open subset of $\Gamma_a\setminus\left\lbrace\left(\infty,\frac{a+1}{2}\right)\right\rbrace$, which has no singularities. Thus, $\graph(\f)$ is a Riemann surface and $\f$ is holomorphic.

It is clear that $\f^{-1}(w)$ is compact, for all $w\in\Delta_{\J}$, since 
$$1\leq|\f^{-1}(w)|\leq|\F^{-1}(w)|\leq 2.$$
Moreover, $\f$ is a closed map, and therefore a proper map. Indeed, let $C\subset\F^{-1}(\Delta'_{\J})$ be closed. Then $C=C'\cap\F^{-1}(\Delta'_{\J})$ for some closed subset $C'\subset\widehat{\mathbb{C}}$. Observe that $\f(C)=\F(C')\cap\Delta'_{\J}$ and that
$$\F(C')=\pi_2\left(\Gamma_a\cap(C'\times\widehat{\mathbb{C}})\right).$$
By compactness of $\widehat{\mathbb{C}}$, $\pi_2$ is closed, and since both $\Gamma_a$ and $C'\times\widehat{\mathbb{C}}$ are closed subsets of $\widehat{\mathbb{C}}\times\widehat{\mathbb{C}}$, then $\F(C')$ is closed. Therefore $\f(C)$ is closed in $\Delta'_{\J}$ and $\f$ is a closed map. We conclude $\f$ is proper. Moreover, by definition of $\f$, every preimage $z$ of a point $w\in\Delta'_{\J}$ belongs to the domain of $\f$, and $\f(z)=w$. Therefore $\f$ has degree $2$, since $\F$ has $2$ preimages of a generic point in $\Delta'_{\J}$.

We proceed to show part 2. Let $z_0$ be a critical point of $\f$. Then $(z_0,\f(z_0))$ belongs to
$$A_2(\Gamma_a)\cap(\F^{-1}(\Delta'_{\J})\times\Delta'_{\J}).$$ 
Observe that $\F^{-1}(\Delta'_{\J})\subset\Delta'_{\J}$ and $1\notin\Delta'_{\J}$, so $(1,\J(-2))\notin\F^{-1}(\Delta'_{\J})\times\Delta'_{\J}$. Since we also have that $\infty$ and $\J(\infty)$ lie on opposite sides of $\partial\Delta_{\J}$ and from Proposition \ref{b2}, it must be the case that $(z_0,\f(z_0))=(-1,\J(2))$. Thus, $z_0=-1$ and $\J(2)\in\Delta'_{\J}$, and therefore $2\in\widehat{\mathbb{C}}\setminus\overline{\Delta'_{\J}}$.
To prove the reverse implication, note that whenever $2\in\widehat{\mathbb{C}}\setminus\overline{\Delta'_{\J}}$, we have that $\J(2)=\frac{2}{3-a}\in\Delta'_{\J}$, so there exists $z_0\in\F^{-1}(\Delta'_{\J})$ such that $\f(z_0)=\frac{2}{3-a}$, by definition of $\f$. Then $z_0=-1\in\F^{-1}(\Delta'_{\J})$ is a critical point for $\f$, as $\f$ has degree $2$ and $\f^{-1}\left(\frac{2}{3-a}\right)=\lbrace -1\rbrace$. This proves that $\f$ has a critical point if and only if $2\in\widehat{\mathbb{C}}\setminus\overline{\Delta'_{\J}}$, and the critical point is $z_0=-1$.

Finally, we prove part 3. Observe that if $w\in\f^{-1}(z)$, then $[ w]=[ z]$, so $\f^{-1}(\mathcal{E}_{a,-})\subset\mathcal{E}_{a,-}$ and thus, $|\f^{-1}(\mathcal{E}_{a,-})|\leq|\mathcal{E}_{a,-}|$. In particular, all points of $\mathcal{E}_{a,-}$ must have only one preimage under $\f$, and therefore, be critical. By part 2, if $\mathcal{E}_{a,-}\neq\emptyset$, then $-1$ must be fixed.
We know $\f(-1)=\frac{2}{3-a}$, and therefore if $\mathcal{E}_{a,-}\neq\emptyset$, then $\frac{2}{3-a}=-1$. Solving the equation we get that $a=5$. On the other hand, if $a=5$, then $2\in\widehat{\mathbb{C}}\setminus\Delta'_{\J}$ and therefore $z_0=-1$ is a critical point for $\f$ that is fixed.
\end{proof}

\begin{remark} Put $a\neq 1$, $|a-4|\leq 3$, and $(\Delta_{\cov},\Delta_{\J})$ the Klein combination pair from \cite{BulLom20I} described in Section \ref{thefamily}. Away from $z=1$, $\Delta'_{\J}$ agrees with $\Delta_{\J}$. Let $r$ be the radius of the circle $\partial\Delta_{\J}$. Since the circle passes through $a$ and has center $(1+r)$, then $2\in\widehat{\mathbb{C}}\setminus\overline{\Delta_{\J}}$ if and only if $r>\frac{1}{2}$. Equivalently, $a\notin\overline{B}\left(\frac{3}{2},\frac{1}{2}\right)$. Since $2$ is away from $z=1$, $2\in\widehat{\mathbb{C}}\setminus\overline{\Delta'_{\J}}$ if and only if $a\notin\overline{B}\left(\frac{3}{2},\frac{1}{2}\right)$ as well.
\begin{itemize}
\item if $0<r\leq\frac{1}{2}$, then $\f:\F^{-1}(\Delta'_{\J})\to\Delta'_{\J}$ has no critical points and hence, is an unramified covering map. Since $\f$ is a 2-to-1 map, the Riemann-Hurwitz formula yields
$$\chi(\F^{-1}(\Delta'_{\J}))=2\chi(\Delta'_{\J})=2$$
where $\chi$ denotes the Euler characteristic. Since $\f$ has degree 2, $\F^{-1}(\Delta'_{\J})$ has at most $2$ connected components, each of them with Euler characteristic at most $1$. It follows that $\F^{-1}(\Delta'_{\J})$ has exactly $2$ connected components, each of them homeomorphic to a disk.
\item if $\frac{1}{2}<r<3$, then $\f$ is a ramified covering of degree $2$ with one critical point at $z_0=-1$, whose ramification index equals 2. Therefore, again by the Riemann-Hurwitz formula, we get that
$$\chi(\F^{-1}(\Delta'_{\J}))=2\chi(\Delta'_{\J})-(2-1)=1.$$
Let $\Omega$ be the component of $\F^{-1}(\Delta'_{\J})$ that contains $-1$. Then $\restr{f}{\Omega}$ is not locally injective at $-1$, and therefore of degree $2$. Therefore $\F^{-1}(\Delta'_{\J})=\Omega$ has one connected component, which is homeomorphic to a circle, and mapped 2-to-1 onto $\Delta'_{\J}$.
\end{itemize}
\end{remark}

We finish this section by listing some properties of the periodic points of $\F$. In order to do so, recall that
$$\Per_n(F)=\pi_1(\Gamma^{(n)}\cap\mathfrak{D}_{X}),$$
where $\Gamma^{(n)}$ is the graph of $F^n$ and $\mathfrak{D}_{X}$ is the diagonal in $X\times X$. Note that if $F$ is a correspondence on $X$,
 $$\Per_n(F)=\lbrace z\in X|z\in F^n(z)\rbrace=\lbrace z\in X|z\in F^{-n}(z)\rbrace.$$

\begin{lemma}\label{periodic} 
For every $a\in\mathcal{K}$, we have that
\begin{itemize}
\item[1.] $\Per_n(\F)\subset\Lambda_{a,-}\cup\Lambda_{a,+}$,
\item[2.] $\J(\Per_n(\F))=\Per_n(\F)$,
\item[3.] $\J(\Per_n(\F)\cap\Lambda_{a,\pm})=\Per_n(\F)\cap\Lambda_{a,\mp}$, and
\item[4.] $\Per_n(\F)\cap\Lambda_{a,-}=\Per_n(\f)$.
\end{itemize}
\end{lemma}

\begin{proof}
In order to prove part $1$, suppose $z\in\Per_n(\F)$. If $z\in\overline{\Delta'_{\J}}$, then $z\in\F^{-kn}(z)\subset\F^{-kn}(\overline{\Delta'_{\J}})$ for all $k\geq 1$. Since the sets $\F^{-k}(\overline{\Delta'_{\J}})$ are nested and their intersection is $\Lambda_{a,-}$, then $z\in\Lambda_{a,-}$. On the other hand, if $z\notin\Lambda_{a,-}$, Lemma \ref{lambdas} part $3$, there exists $m\geq 1$ so that $\F^m(z)\subset\widehat{\mathbb{C}}\setminus\Delta'_{\J}$. A similar argument as above shows that for all $k$ sufficiently large, $z\in\F^k(\widehat{\mathbb{C}}\setminus\overline{\Delta'_{\J}})$, which implies that $z\in\Lambda_{a,+}$.

Next, note that if $z\in\Per_n(\F)$, then $z\in\F^{-n}(z)$ and $\J(z)\in\J(\F^{-n}(z))=\F^n(\J(z))$. Therefore $\J(z)\in\Per_n(\F)$ as well. Since $\J$ is an involution, this shows part $2$.

Recall from Remark \ref{lambdas} that $\J$ sends the limits sets to each other. Therefore, $\J(\Per_n(\F)\cap\Lambda_{a,-})=\Per_n(\F)\cap\Lambda_{a,+}$, and again since $\J$ is an involution, we also have that 
$$\J(\Per_n(\F)\cap\Lambda_{a,+})=\Per_n(\F)\cap\Lambda_{a,-}.$$
This proves part $3$.

Finally, we prove part $4$. We have that for each $z\in\F^{-1}(\Delta'_{\J})\cup\lbrace 1\rbrace$, $\F^{-n}(z)=\f^{-n}(z)$ and part $1$ implies that all periodic points of $\F$ not in $\Lambda_{a,+}$ must belong to $\Lambda_{a,-}$. Since $z\in\Per_n(\f)$ if and only if $z\in\f^{-n}(z)=\F^{-1}(z)$, then
$$\Per_n(\F)\cap\Lambda_{a,-}=\Per_n(\F)\cap(\F^{-1}(\Delta'_{\J})\cup\lbrace 1\rbrace)=\Per_n(\f).$$

\end{proof}

%% file: Equid.tex
In this section we prove theorems \ref{B} and \ref{C} using the results from \cite{Lju83} and \cite{FreLopMan83}, together with \cite{BulLom20I} and \cite{BulLom20II}.

Recall every $P_A$ of the form
\begin{equation}\label{parabolic quadratic}
P_A(z)=z+1/z+A,
\end{equation}
with $A\in\mathbb{C}$, has critical points $\pm 1$ and a parabolic fixed point at $\infty$ with multiplier equal to $1$. Now note that in the coordinate $\phi(z)=1/z$ we have that 
$$P_A^{\phi}(z)=\frac{z}{z^2+Az+1}=z-Az^2+(A^2-1)z^3-\dots$$
near $0$. Thus, $z=0$ is a fixed point of $P_A^{\phi}$ with multiplicity $2$ if $A\neq 0$, and $3$ if $A=0$. We conclude that for $A\neq 0$, the parabolic fixed point has multiplicity $1$, and $3$ for $A=0$.

For $A\in\mathbb{C}\setminus\lbrace 0\rbrace$, let $\Omega_A$ be the basin of attraction of $\infty$, and for $A=0$, we make the choice $\Omega_0=\lbrace x+iy| x,y\in\mathbb{R},x>0\rbrace$. We define the \emph{filled Julia set} of $P_A$ as the set
$$K_{P_A}=\widehat{\mathbb{C}}\setminus\Omega_A.$$

\begin{remark}\label{filled} No periodic points of $P_A$ live in $\Omega_A$. Indeed, it is clear that points in $\Omega_A$ for $A\neq 0$ cannot be periodic, as their iterates converge to $\infty$. Moreover, it is easy to check that if $\Re(z)$ denotes the real part of $z\neq 0$, then $\Re(P_A(z))=\Re(z)\frac{|z|^2+1}{|z|^2}$. Then $\Omega_0$ is completely invariant, and has no periodic points.  Therefore, for all $A\in\mathbb{C}$ we have that $\Per_n(P_A)\subset K_{P_A}$, for all $n\geq 1$.
\end{remark}

\begin{lemma}\label{mu}
For every $a\in\mathcal{K}$, there exists a measure $\mu_-$ with $\supp(\mu_-)=\partial\Lambda_{a,-}$ such that
$$\frac{1}{2^n}(\f^n)^*\delta_{z_0}\to\mu_-$$
weakly, for all $z_0\in\Delta_{\J}\setminus\mathcal{E}_{a,-}$.
\end{lemma}
Here $\mathcal{E}_{a,-}$ is the exceptional set of the two-sided restriction $\f$ of $\F$ that leaves $\Lambda_{a,-}$ invariant, described in the Section \ref{exceptional and periodic}.
\begin{proof}
From the previous section, we can see that the two-sided restriction $\f$ can extend around $1$. In \cite[Proposition 5.2]{BulLom20I}, the authors proved the existence of closed topological disks $V'_a$ and $V_a$ containing $\Lambda_{a,-}$ with $V'_a=\F^{-1}(V_a)\subset V_a$ and satisfying most properties for $\f$ to be a parabolic-like map from $V'_a$ and $V_a$. Moreover, from the proof of  \cite[Theorem B]{BulLom20I}, there exists a neighborhood $U$ of $\Lambda_{a,-}$ and a quasiconformal map $h:U\cap V'_a\to h(U)\cap V_a$ such that $h\circ f=P_A\circ h$, where $P_A$ is as in equation (\ref{parabolic quadratic}). Such conjugacy sends $\Lambda_{a,-}$ onto the filled Julia set $K_{P_A}$ of $P_A$.

From \cite{FreLopMan83} and \cite{Lju83}, there exists a measure $\tilde{\mu}_A$ on $\widehat{\mathbb{C}}$ supported on the Julia set $\mathcal{J}_{P_A}$ of $P_A$, such that for all $z_0\in\widehat{\mathbb{C}}$ not in the exceptional set $E$ of $P_A$, $\frac{1}{2^n}(P_A^n)^*\delta_{z_0}$ is weakly convergent to $\tilde{\mu}_A$.

Recall $\Lambda_{a,-}$ is the intersection of the nested compact sets $\F^{-1}(\overline{\Delta'_{\J}})$. Thus, there exists $N\in\mathbb{N}$ such that $\F^n(\overline{\Delta'_{\J}})\subset U\cap V'_a$, for all $n\geq N$. Observe as well that $h^{-1}(E)=\mathcal{E}_{a,-}$ as elements in $\mathcal{E}_{a,-}$ must have only one preimage and $h$ is a homeomorphism, and the analogous holds for $E$ and $h^{-1}$.

Now take $z_0\in\overline{\Delta'_{\J}}\setminus\mathcal{E}_{a,-}$. Then $\f^{-N}(z_0)=\F^{-N}(z_0)\subset U\cap V'_a$. Then for each $\zeta\in \f^{-N}(z_0)$, we have that $h(\zeta)\notin E$ and $\frac{1}{2^n}(P_A^n)^*\delta_{h(\zeta)}$ is weakly convergent to $\tilde{\mu}_A$. Note that $\delta_{h(\zeta)}=h_*\delta_{\zeta}$, and therefore for $n\geq N$,
$$\frac{1}{2^n}(P_A^n)^*\delta_{h(\zeta)}=\frac{1}{2^n}(h^{-1}\circ P_A^n)^*\delta_{\zeta}=\frac{1}{2^n}(\f^n\circ h^{-1})^*\delta_{\zeta}=h_*\left(\frac{1}{2^n}(\f^n)^*\delta_{\zeta}\right).$$
Since the left hand side is weakly convergent to $\tilde{\mu}_A$, then $\frac{1}{2^n}(\f^n)^*\delta_{\zeta}$ is weakly convergent to $\mu_-\coloneqq h^*\tilde{\mu}_A$, which is supported on $h(\mathcal{J}_{P_A})=\partial\Lambda_{a,-}$. Thus,
$$\frac{1}{2^n}(\f^n)^*\delta_{z_0}=\frac{1}{2^N}\sum\limits_{\zeta\in \f^{-N}(z_0)}\nu_{\f^N}(\zeta)\frac{1}{2^{n-N}}(\f^{n-N})^*\delta_{\zeta}$$
is weakly convergent to $\mu_-$ as desired.
\end{proof}

\begin{proof}[Proof of Theorem \ref{B}] Let $\mu_-$ be as in Lemma \ref{mu} and put $\m\coloneqq(\J)_*\mu_-$. Clearly $\supp(\m)=\partial\Lambda_{a,+}$. Denote by $\delta_{z_0}$ the Dirac measure at $z_0$. If $z_0\in\widehat{\mathbb{C}}\setminus(\Lambda_{a,-}\cup\E)$, there exists $n_0\in\mathbb{Z}^+\cup\lbrace 0\rbrace$ such that ${\F^{n_0}(z_0)\cap\overline{\Delta'_{\J}}=\emptyset}$ by part $3$ of Lemma \ref{lambdas}. In particular, this gives us that $\F^{n_0}(z_0)$ is contained in $\widehat{\mathbb{C}}\setminus\Delta'_{\J}$, and therefore ${\F^n(z_0)\subset\widehat{\mathbb{C}}\setminus\Delta'_{\J}}$, for all ${n\geq n_0}$. In addition, for every $z_j\in\F^{n_0}(z_0)$, $\J(z_j)\notin\mathcal{E}_{a,-}$. Thus, for such a $z_0$, we get that 
\begin{eqnarray*}
\frac{1}{2^n}(\F^n)_*\delta_{z_0}&=&\frac{1}{2^n}\sum\limits_{\zeta_j\in\F^{n_0}(z_0)}\nu_{\restr{\pi_1}{\Gamma^{(n_0)}}}(z_0,\zeta_j)(\F^{n-n_0})_*\delta_{\zeta_j}\\
&=&\frac{1}{2^{n_0}}\sum\limits_{\zeta_j\in\F^{n_0}(z_0)}\nu_{\restr{\pi_1}{\Gamma^{(n_0)}}}(z_0,\zeta_j)\frac{1}{2^{n-n_0}}(\F^{n-n_0})_*\delta_{\zeta_j}\to\mu_{\infty}\\
&=&\frac{1}{2^{n_0}}\sum\limits_{\zeta_j\in\F^{n_0}(z_0)}\nu_{\restr{\pi_1}{\Gamma^{(n_0)}}}(z_0,\zeta_j)\frac{1}{2^{n-n_0}}(\J\circ\F^{-(n-n_0)}\circ\J)_*\delta_{\zeta_j}\\
&=&\frac{1}{2^{n_0}}\sum\limits_{\zeta_j\in\F^{n_0}(z_0)}\nu_{\restr{\pi_1}{\Gamma^{(n_0)}}}(z_0,\zeta_j)(\J)_*\left(\frac{1}{2^{n-n_0}}(\F^{n-n_0})^*\delta_{\J(\zeta_j)}\right)\\
&\to&\frac{1}{2^{n_0}}\sum\limits_{\zeta_j\in\F^{n_0}(z_0)}\nu_{\restr{\pi_1}{\Gamma^{(n_0)}}}(z_0,\zeta_j)(\J)_*\mu_-=\m,
\end{eqnarray*}
weakly, as $n\to\infty$, where $\Gamma^{(n_0)}$ is the graph of $\F^{n_0}$.

Now observe that the two-sided restriction $\f$ of $\F$ sends $\Lambda_{a,-}$ to itself. Indeed, note that the sets $\F^{-n}(\overline{\Delta'_{\J}})$ are nested and if $\f(z)\notin\F^{-n}(\overline{\Delta'_{\J}})$ for some $n\in\mathbb{Z}^+$, then $z\notin\F^{n-1}(\overline{\Delta'_{\J}})\supset\Lambda_{a,-}$. Note as well that $-2\in\F^{-1}(1)$ and $1\in\Lambda_{a,-}$, so $-2\in\Lambda_{a,-}$. Define $\tilde{\f}:\Lambda_{a,-}\to(\widehat{\mathbb{C}}\setminus\overline{\Delta'_{\J}})\cup\lbrace 1\rbrace$ by $\tilde{\f}(z)=w$, where $\F(z)\setminus\lbrace \f(z)\rbrace=\lbrace w\rbrace$ for $z\neq -2$, and $\tilde{\f}(-2)=\f(-2)=1$ for $z\in\F^{-1}(1)\setminus\lbrace 1\rbrace$.
Then for $z_0\in\Lambda_{a,-}\setminus\E$ we have that $\F(z_0)=\lbrace \f(z_0),\tilde{\f}(z_0)\rbrace$, and
\begin{eqnarray*}
\frac{1}{2}(\F)_*\delta_{z_0} &=& \frac{1}{2}\delta_{\f(z_0)}+\frac{1}{2}\delta_{\tilde{\f}(z_0)}\\
\frac{1}{4}(\F^2)_*\delta_{z_0} &=& \frac{1}{4}\delta_{{\f}^2(z_0)}+\frac{1}{4}\delta_{\tilde{\f}\circ \f(z_0)}+\frac{1}{2}(\F)_*\delta_{\tilde{\f}(z_0)}\\
\frac{1}{8}(\F^3)_*\delta_{z_0} &=& \frac{1}{8}\delta_{{\f}^3(z_0)}+\frac{1}{8}\delta_{\tilde{\f}\circ {\f}^2(z_0)}+\frac{1}{4}(\F)_*\delta_{\tilde{\f}\circ \f(z_0)}+\frac{1}{2}(\F^2)_*\delta_{\tilde{\f}(z_0)}\\
& \vdots & \\
\frac{1}{2^n}(\F^n)_*\delta_{z_0} &=& \frac{1}{2^n}\delta_{{\f}^n(z_0)}+\sum\limits_{j=1}^n \frac{1}{2^j}(\F^{n-j})_*\delta_{\tilde{\f}\circ {\f}^{j-1}(z_0)}.
\end{eqnarray*}

Since $\frac{1}{2^n}\delta_{{\f}^n(z_0)}$ has mass $\frac{1}{2^n}$, then it is weakly convergent to a measure with zero mass. Put 
$$\mu_{n,j}\coloneqq\frac{1}{2^j}(\F^{n-j})_*\delta_{\tilde{\f}\circ {\f}^{j-1}(z_0)}\mbox{ and }\mu_n\coloneqq\frac{1}{2^n}(\F^n)_*\delta_{z_0} - \frac{1}{2^n}\delta_{{\f}^n(z_0)}=\sum\limits_{j=1}^n\mu_{n,j}.$$ If we truncate the sum defining $\mu_n$ at $N<n$, we obtain a sequence $\mu_n^{(N)}=\sum\limits_{j=1}^N\mu_{n,j}$ satisfying $\mu_n^{(N)}\to\sum\limits_{j=1}^N\frac{1}{2^j}\m=(1-2^{-N})\m$ weakly. Let $\varphi:\widehat{\mathbb{C}}\to\mathbb{R}$ be continuous and $\varepsilon>0$. Choose $N\in\mathbb{Z}^+$ big so that $\frac{3}{2^N}\sup|\varphi|<\varepsilon$, and $n>N$ such that 
$$\Big|\int\varphi d\mu_n^{(N)}-(1-2^{-N})\int\varphi d\m\Big|<\frac{1}{2^N}\sup|\varphi|$$
by the convergence $\mu_n^{(N)}\to(1-2^{-N})\m$. Then
\begin{eqnarray*}
\Big|\int\varphi d\mu_n^{(N)}-\int\varphi d\m\Big|&\leq&\Big|\int\varphi d\mu_n^{(N)}-\left(1-\frac{1}{2^N}\right)\int\varphi d\m\Big|+\frac{1}{2^N}\Big|\int\varphi d\m\Big|\\
&\leq&\frac{1}{2^{N-1}}\sup|\varphi|
\end{eqnarray*}

Since $\mu_n-\mu_n^{(N)}=\sum\limits_{j=N+1}^n\mu_{n,j}$ has mass at most $\frac{1}{2^N}$, we get
\begin{eqnarray*}
\Big|\int\varphi d\mu_n-\int\varphi d\m\Big|&\leq&\Big|\int\varphi d\mu_n-\int\varphi d\mu_n^{(N)}\Big|+\Big|\int\varphi d\mu_n^{(N)}-\int\varphi d\m\Big|\\
&\leq&\frac{3}{2^N}\sup|\varphi|<\varepsilon.
\end{eqnarray*}
This proves that for all $\varphi:\widehat{\mathbb{C}}\to\mathbb{R}$, we have that $\int\varphi d\mu_n\to\int\varphi d\m$ as $n\to\infty$ and hence $\mu_n\to\m$. Since $\frac{1}{2^n}(\F^n)_*\delta_{z_0}=\frac{1}{2^n}\delta_{{\f}^n(z_0)}+\mu_n$, we obtain that $\frac{1}{2^n}(\F^n)_*\delta_{z_0}$ is weakly convergent to $\m$, as desired.

Applying again the push-forward by $\J$ we get that for every $z_0\notin\E$, $\frac{1}{2^n}(\F^n)^*\delta_{z_0}$ is weakly convergent to $\mu_-$.
\end{proof}

Finally, we show the asymptotic equidistribution of periodic points of $\F$ of order $n$ with respect to the probability measure $\frac{1}{2}(\mu_-+\mu_+)$.

\begin{proof}[Proof of Theorem \ref{C}]

From Lemma \ref{periodic}, we have that all periodic points lie in $\Lambda_{a,-}\cup\Lambda_{a,+}$ and
$$|\Per_n(\F)\cap\Lambda_{a,-}|=|\Per_n(\F)\cap\Lambda_{a,+}|=|\Per_n(\f)|.$$
Denote this number by $d_n$. Since $1\in\F(1)$ and Lemma \ref{periodic} says $\Lambda_{a,-}\cap\Lambda_{a,+}=\lbrace 1\rbrace$, then 
$$(\Per_n(\F)\cap\Lambda_{a,-})\cap(\Per_n(\F)\cap\Lambda_{a,+})=\lbrace 1\rbrace.$$
Thus, $|\Per_n(\F)|=2d_n-1$. Since the conjugacy $h$ between $\f$ and the quadratic rational map $P_A$ sends $\Lambda_{a,-}$ onto the filled Julia set $K_{P_A}$ of $P_A$, which contains all periodic points of $P_A$ by Remark \ref{filled}. Thus, $\Per_n(P_A)=h(\Per_n(\F)\cap\Lambda_{a,-})$, and 
$$|\Per_n(\f)|=|\Per_n(P_A)|=d_n$$
 as well. We have that $\lim\limits_{n\to\infty}d_n=\infty$ (see page 363 in \cite{Lju83}). Then

\begin{eqnarray*}
\frac{1}{2d_n-1}\sum\limits_{z\in\Per_n(\F)\cap\Lambda_{a,-}}\delta_z&=&\frac{1}{2d_n-1}\sum\limits_{z\in\Per_n(\f)}\delta_z\\
&=&h^*\left(\frac{1}{2d_n-1}\sum\limits_{z\in\Per_n(\f)}\delta_{h(z)}\right)\\
&=&h^*\left(\frac{d_n}{2d_n-1}\frac{1}{d_n}\sum\limits_{\zeta\in\Per_n(P_A)}\delta_{\zeta}\right),
\end{eqnarray*}
which is weakly convergent to $\frac{1}{2}h^*\widetilde{\mu}_A=\frac{1}{2}\mu_-$ by  Corollary to Theorem 3 in \cite{Lju83}. Thus,

$$\frac{1}{2d_n-1}\sum\limits_{z\in\Per_n(\F)\cap\Lambda_{a,+}}\delta_z=\J^*\left(\frac{1}{2d_n-1}\sum\limits_{\J(z)\in\Per_n(\F)\cap\Lambda_{a,-}}\delta_{\J(z)}\right)$$
is weakly convergent to $\frac{1}{2}\J^*\mu_-=\frac{1}{2}\mu_+$, and
$$\frac{1}{|\Per_n(\F)|}\sum\limits_{z\in \Per_n(\F)}\delta_z=\frac{1}{2d_n-1}\sum\limits_{\substack{z\in\Per_n(\F)\\z\in\Lambda_{a,-}}}\delta_z+\frac{1}{2d_n-1}\sum\limits_{\substack{z\in\Per_n(\F)\\z\in\Lambda_{a,+}}}\delta_z-\frac{1}{2d_n-1}\delta_1$$  converges weakly to $\frac{1}{2}(\mu_-+\mu_+)$, as desired.

Finally, we proceed to show that the equidistribution still holds counting with multiplicity. Observe that  the only points of $\Lambda_{a,-}$ that are not in the interior of the topological conjugacy $h$ are $-2$ and $1$, and we have that $-2$ is not a periodic point for $\F$. We claim that for every $z_0\in\Per_n(f)\setminus\lbrace 1\rbrace$, we have that 
$$\nu_{\restr{\pi_1}{\graph(\f^n)\cap\mathfrak{D}_{\widehat{\mathbb{C}}}}}(z_0,z_0)=\nu_{\restr{\pi_1}{\graph(P^n_A)\cap\mathfrak{D}_{\widehat{\mathbb{C}}}}}(h(z_0),h(z_0)).$$
Indeed, the topological conjugacy $h$ sends attracting periodic points to attracting periodic points, so if $|(\f^n)'(z_0)|<1$, then $|(P^n_A)'(h(z_0))|<1$ as well. Thus, both $z_0$ and $h(z_0)$ have multiplicity $1$ as a periodic point in this case. The same happens for $|(\f^n)'(z_0)|>1$. Finally, if $|(\f^n)'(z_0)|=|(P^n_A)'(h(z_0))|=1$, Naishul's Theorem \cite{Nai82} (later re-proven by P\'erez-Marco in \cite{Per97}) shows that $(\f^n)'(z_0)=(P^n_A)'(h(z_0))$. Moreover, attracting directions are preserved under topological conjugacy, so the parabolic fixed point $z_0$ of $\F$ and the parabolic fixed point $h(z_0)$ of $P_A$ have the same number of attracting directions. Therefore, both $z_0$ and $h(z_0)$ have the same multiplicity as periodic points of order $n$ of $\f$ and $P_A$, respectively. Therefore,

\begin{eqnarray*}
\sum\limits_{\substack{z\in\Per_n(\F)\cap\Lambda_{a,-} \\ z\neq 1}}\nu_{\restr{\pi_1}{\Gamma_a^{(n)}\cap\mathfrak{D}_{\widehat{\mathbb{C}}}}}(z,z)&=&\sum\limits_{\substack{\zeta\in\Per_n(P_A)\\\zeta\neq\infty}}\nu_{\restr{\pi_1}{\graph(P^n_A)\cap\mathfrak{D}_{\widehat{\mathbb{C}}}}}(\zeta,\zeta)\\
&=&(2^n+1)-\nu_{\restr{\pi_1}{\graph(P^n_A)\cap\mathfrak{D}_{\widehat{\mathbb{C}}}}}(\infty,\infty).
\end{eqnarray*}

From Lemma \ref{periodic} part $2$ and given that $\J$ is a Mobius transformation, we have that
$$\sum\limits_{\substack{z\in\Per_n(\F)\cap\Lambda_{a,+} \\ z\neq 1}}\nu_{\restr{\pi_1}{\Gamma_a^{(n)}\cap\mathfrak{D}_{\widehat{\mathbb{C}}}}}(z,z)=(2^n+1)-\nu_{\restr{\pi_1}{\graph(P^n_A)\cap\mathfrak{D}_{\widehat{\mathbb{C}}}}}(\infty,\infty)$$
as well.

Since the multiplicity of a periodic point is the multiplicity as a periodic point of its minimal period and since $\phi^{-1}(0)=\infty$, we have that
$$\nu_{\restr{\pi_1}{\graph(P^n_A)\cap\mathfrak{D}_{\widehat{\mathbb{C}}}}}(\infty,\infty)=\left\lbrace\begin{matrix}3& \mbox{ if }A=0\\2&\mbox{ otherwise}\end{matrix}.\right.$$
Similarly, from the equations for $g^{\psi}$ listed in Remark \ref{taylor}, we have that
$$\nu_{\restr{\pi_1}{\Gamma_a^{(n)}\cap\mathfrak{D}_{\widehat{\mathbb{C}}}}}(1,1)=\left\lbrace\begin{matrix}
4& \mbox{ if }a=7\hspace{0.75cm}\\
2& \mbox{ if }a\in\mathcal{K}\setminus\lbrace 7\rbrace.\end{matrix}\right.$$

In \cite[Corollary 4.3]{BulLom20II} the authors showed that for $a=7$, the member of the family of quadratic rational maps
$$\left\lbrace P_A(z)=z+1/z+A|A\in\mathbb{C}\right\rbrace$$
that is conjugate to $\f$ is $P_0$. Therefore, for all $a\in\mathbb{C}\setminus\lbrace 1\rbrace$, $|a-4|\leq 3$,

\begin{eqnarray*}
\sum\limits_{z\in\Per_n(\F)}\nu_{\restr{\pi_1}{\Gamma_a^{(n)}\cap\mathfrak{D}_{\widehat{\mathbb{C}}}}}(z,z)&=&2\left(2^n+1-\nu_{\restr{\pi_1}{\graph(P^n_A)\cap\mathfrak{D}_{\widehat{\mathbb{C}}}}}(\infty,\infty)\right)+\nu_{\restr{\pi_1}{\Gamma_a^{(n)}\cap\mathfrak{D}_{\widehat{\mathbb{C}}}}}(1,1)\\&=&2^{n+1}
\end{eqnarray*}
and we have that
$$\frac{1}{2^n}\sum\limits_{\substack{z\in\Per_n(\F)\cap\Lambda_{a,-}\\z\neq 1}}\nu_{\restr{\pi_1}{\Gamma_a^{(n)}\cap\mathfrak{D}_{\widehat{\mathbb{C}}}}}(z,z)\delta_z=h^*\left(\frac{1}{2^n}\sum\limits_{\substack{\zeta\in\Per_n(P_A)\\\zeta\neq\infty}}\nu_{\restr{\pi_1}{\graph(P^n_A)\cap\mathfrak{D}_{\widehat{\mathbb{C}}}}}(\zeta,\zeta)\delta_{\zeta}\right),$$
which also converges weakly to $\mu_-$ by \cite[Theorem 3]{Lju83}. Using that $\J$ preserves multiplicities and that $\J^*\mu_-=\mu_+$, we have that
$$\frac{1}{2^n}\sum\limits_{\substack{z\in\Per_n(\F)\cap\Lambda_{a,+}\\z\neq 1}}\nu_{\restr{\pi_1}{\Gamma_a^{(n)}\cap\mathfrak{D}_{\widehat{\mathbb{C}}}}}(z,z)\delta_z=\J^*\left(\frac{1}{2^n}\sum\limits_{\substack{z\in\Per_n(\F)\cap\Lambda_{a,-}\\z\neq 1}}\nu_{\restr{\pi_1}{\Gamma_a^{(n)}\cap\mathfrak{D}_{\widehat{\mathbb{C}}}}}(z,z)\delta_z\right)$$
is weakly convergent to $\mu_+$. Since $\dfrac{1}{2^n}\nu_{\restr{\pi_1}{\Gamma_a^{(n)}\cap\mathfrak{D}_{\widehat{\mathbb{C}}}}}(1,1)\delta_1$ has total mass $\dfrac{1}{2^n}$ and
$$\Per_n(\F)=\lbrace 1\rbrace\cup\left(\Per_n(\F)\cap\Lambda_{a,-}\setminus\lbrace 1\rbrace\right)\cup\left(\Per_n(\F)\cap\Lambda_{a,+}\setminus\lbrace 1\rbrace\right)$$
is a disjoint union, then
$$\frac{1}{2^{n+1}}\sum\limits_{z\in\Per_n(\F)}\nu_{\restr{\pi_1}{\Gamma_a^{(n)}\cap\mathfrak{D}_{\widehat{\mathbb{C}}}}}(z,z)\delta_z$$
is weakly convergent to $\frac{1}{2}(\mu_-+\mu_+)$.
\end{proof}

In what follows we will introduce terminology and prove Theorem \ref{D}.
\\
\begin{definition} A holomorphic correspondence $F$ with graph $\Gamma$ is said to be \emph{postcritically finite} if for all $w\in B_2(\Gamma)$, there exist $0\leq m<n$ so that
\begin{equation}\label{preperiodic}
F^m(w)\cap F^n(w)\neq\emptyset.
\end{equation}
\end{definition}

Equivalently, for every $w\in B_2(\Gamma)$, there exists $m\geq 0$ and $z\in F^m(w)$ so that $z\in\Per_n(F)$ for some $n\geq 1$. This definition generalizes that of postcritically finite rational maps. Indeed, if $F$ is a rational map that is postcritically finite as a correspondence, then we have that every $w\in B_2(\Gamma)$ is pre-periodic. Since $B_2(\Gamma)=F(\critpt(F))$, then every critical point is pre-periodic as well. Hence $F$ is postcritically finite as a rational map.

If $F$ is a holomorphic correspondence and $z_0\in\Per_n(F)$, there exists a cycle
\begin{equation}\label{cycle}
(z_0,z_1,z_2,\cdots,z_{n-1})
\end{equation}
so that $z_i\in F(z_{i-1})$ and $z_0\in F(z_{n-1})$. 

\begin{definition} We say that a holomorphic correspondence with graph $\Gamma$ is \emph{superstable} if there exists $\alpha\in A_2(\Gamma)$ and a cycle (\ref{cycle}) satisfying that $z_0=\pi_1(\alpha)$ and $z_1=\pi_2(\alpha)$. 
\end{definition}

Observe that if $F$ is a rational map, then $\pi_1(\alpha)$ corresponds to a critical point and $\pi_2(\alpha)$ corresponds to its critical value. Therefore the definition of superstable correspondences reduces to $F$ having a superstable cycle, meaning a cycle containing a critical point.

In the context of the family $\lbrace\F\rbrace_{a\in\mathcal{K}}$, Proposition \ref{b2} says that 
$$B_2(\Gamma_a)=\left\lbrace \frac{a+1}{2},\frac{4a+2}{a+5},\frac{2}{3-a}\right\rbrace.$$ 
Observe that $\frac{a+1}{2}=\J(\infty)\notin\Lambda_{a,-}\cup\Lambda_{a,+}$. Indeed, since $\infty$ is fixed by $\cov$ and $(\Delta'_{\cov},\Delta'_{\J})$ is a Klein combination pair, then $\infty\in(\widehat{\mathbb{C}}\setminus\Delta'_{\cov})\setminus\lbrace 1\rbrace\subset\Delta'_{\J}$, so $\J(\infty)\notin\Lambda_{a,-}\subset\overline{\Delta'_{\J}}$. On the other hand, $\F^{-1}(\J(\infty))=\lbrace\infty\rbrace$ does not intersect $\Lambda_{a,+}$, so $\J(\infty)\notin\Lambda_{a,+}$. By Lemma (\ref{periodic}), $\frac{a+1}{2}$ cannot satisfy (\ref{preperiodic}). Therefore there are no parameters $a\in\mathcal{K}$ for which $\F$ is postcritically finite.

On the other hand, $\F^{-(n+1)}(\frac{4a+2}{a+5})=\F^{-n}(1)\subset\Lambda_{a,-}$ and $\frac{4a+2}{a+5}\in\Lambda_{a,+}$, so $\frac{4a+2}{a+5}$ is not periodic. Therefore, $\F$ is superstable if and only if $\frac{2}{3-a}$ is periodic. Since $\F^{-1}\left(\frac{2}{3-a}\right)=\lbrace -1\rbrace$ and $\F^{-1}=\f^{-1}$, then $\F$ is superstable if and only if $-1$ is a critical point of $\f$ that is periodic. We say that the parameter $a\in\mathcal{K}$ is \emph{superstable} whenever $\F$ is superstable.

We consider the quadratic family
$$\lbrace p_c(z)=z^2+c|c\in\mathbb{C}\rbrace$$
and let $\hat{P}_n\coloneqq\lbrace c\in\mathbb{C}|p^n_c(0)=0\rbrace$ be the set of superstable parameters of order $n$. Note that $\hat{P}_n$ has $2^{n-1}$ points, counted with multiplicity, and that it is contained in the Mandelbrot set $\mathcal{M}$.  Levin proved in \cite{Lev90} that 
$$\lim\limits_{n\to\infty}\frac{1}{2^{n-1}}\sum\limits_{c\in\hat{P}_n}\delta_{c}$$
converges to a measure $m_{\bif}$ on $\mathcal{M}$, which we call the \emph{bifurcation measure}.

Let $\mathcal{M}_\Gamma$ denote the connectedness locus of the family $\lbrace\F\rbrace_{a\in\mathcal{K}}$ and and $\mathcal{M}_1$ denote the Parabolic Mandelbrot set. The conjugacy shown in \cite{BulLom20I} between $\f$ and $P_A$ induces a map $a\mapsto 1-A^2$ from $\mathcal{M}_{\Gamma}$ and $\mathcal{M}_1$. From \cite[Main Theorem]{BulLom20II}, that map is a homeomorphism. On the other hand, \cite[Main Theorem]{PetRoe21} says that $\mathcal{M}_1$ is homeomorphic to $\mathcal{M}$. Moreover, both homeomorphisms constructed preserve the type of dynamics associated to the parameter (see \cite[Section 1]{PetRoe21}). Therefore, there is a homeomorphism $\Psi:\mathcal{M}\to\mathcal{M}_\Gamma$ that gives a one-to-one correspondence between superstable parameters of $\mathcal{M}_\Gamma$ and superstable parameters of $\mathcal{M}$. Pushing the bifurcation measure forward through $\Psi$, and writing
$$\hat{P}^{\Gamma}_n=\lbrace a\in\mathcal{K}\mbox{ superstable }|\f^n(-1)=-1\rbrace,$$
we obtain the statement in Theorem \ref{D}.